\documentclass[11pt]{amsart}

 \usepackage[utf8]{inputenc}

 \usepackage[english]{babel}

 \usepackage{amsmath,amsfonts,amssymb}
 \usepackage{mathrsfs}
 \usepackage{dsfont} 
 \usepackage{mathtools} 
 \usepackage{esint} 
 \usepackage{xfrac} 
 \usepackage{defs} 
 \usepackage{fonts} 

 \usepackage{amsthm}
 \usepackage[shortlabels]{enumitem} 
 \usepackage[colorlinks=true,urlcolor=blue, citecolor=red,linkcolor=blue,linktocpage,pdfpagelabels, bookmarksnumbered,bookmarksopen]{hyperref} 
 \usepackage[margin=1in]{geometry} 

 \usepackage[linecolor=blue,backgroundcolor=blue!25,bordercolor=blue]{todonotes}
 \usepackage{cite} 

plus 6pt minus 12pt
\numberwithin{equation}{section}

\newtheorem{theorem}{Theorem}[section]
\newtheorem{proposition}[theorem]{Proposition}
\theoremstyle{definition}
\newtheorem{definition}[theorem]{Definition}
\theoremstyle{plain}
\newtheorem{lemma}[theorem]{Lemma}
\newtheorem{corollary}{Corollary}[theorem]
\theoremstyle{remark}
\newtheorem{remark}[theorem]{Remark}

\title[Self-Improving Properties of Nonlocal Double-Phase Equations]{Self-Improving Inequalities for Bounded Weak Solutions to Nonlocal Double Phase Equations}
\thanks{Support from NSF  DMS-1615726 is gratefully acknowledged.}
 \author{James M. Scott and Tadele Mengesha}
\address[James M. Scott]{Department of Mathematics,
University of Pittsburgh}
\address[Tadele Mengesha]{Department of Mathematics,
University of Tennessee Knoxville, mengesha@utk.edu}
\everymath{\displaystyle}

\begin{document}

\maketitle

\begin{abstract}

We prove higher Sobolev regularity for bounded weak solutions to a class of nonlinear nonlocal integro-differential equations. The leading operator exhibits nonuniform growth, switching between two different fractional elliptic ``phases" that are determined by the zero set of a modulating coefficient. Solutions are shown to improve both in integrability and differentiability. These results apply to operators with rough kernels and modulating coefficients.
To obtain these results we adapt a particular fractional version of the Gehring lemma developed by Kuusi, Mingione, and Sire in their work ``Nonlocal self-improving properties"  Anal. PDE, 8(1):57--114 for the specific nonlinear setting under investigation in this manuscript.
\end{abstract}

\section{Introduction and Main Results}

We are interested in studying regularity properties of weak solutions $u$ to 
\begin{equation}\label{eq:Intro:MainEqn}
\cL u (x) = f(x)\,,
\end{equation}
where for measurable functions $u: \bbR^n \to \bbR$ and for $x \in \bbR^n$ the nonlocal \textit{double phase} operator $\cL$ is defined as
\begin{equation*}
\cL u(x) := \pv \intdm{\bbR^n}{\frac{|u(x)-u(y)|^{p-2}}{|x-y|^{n+sp}} (u(x)-u(y)) + a(x,y) \frac{|u(x)-u(y)|^{q-2}}{|x-y|^{n+tq}} (u(x)-u(y))}{y}\,.
\end{equation*}
Throughout, we assume $n \geq 2$ and the integrability indices $p$, $q$ belong to $(1,\infty)$ with $p \leq q$ and differentiability indices $s$, $t$ belong to $(0,1)$. The abbreviation $\pv$ stands for principal value.
For functions $u$ is smooth enough the operator $\cL$ can be thought of as the sum of a fractional $p$-Laplacian $(-\Delta)^s_p$ and an integro-differential operator whose kernel of differentiability order $t$ and integrability order $q$ is perturbed by the modulating coefficient $a(\cdot,\cdot)$. The order of the operator $\cL$ therefore switches between the fractional elliptic phases $(s,p)$ and $(t,q)$ according to the zero-set of $a(\cdot,\cdot)$.

The operator $\cL$ is a nonlocal analogue of a class of double phase operators of which a prototypical example is given by
\begin{equation*}
\div \left( |\grad u|^{p-2} \grad u) + \alpha(x) |\grad u|^{q-2} \grad u \right) = f\,, \qquad 1 \leq p \leq q\,, \qquad 0 \leq \alpha \leq M < \infty\,.
\end{equation*}
Partial differential equations of the above type arise in the theory of homogenization and elasticity \cite{zhikov1987averaging, Zhikov1995OnLP}. In the event the modulating coefficient $\alpha$ is a positive constant such non-elliptic functionals associated to these operators that exhibit similar $(p,q)$-growth have been treated in the celebrated work of Marcellini \cite{marcellini1993regularity, marcellini1996everywhere, marcellini1991regularity, marcellini1989regularity}. In more recent years there have been considerable efforts to study the regularity of minimizers of $(p,q)$-growth functionals whose integrand depends on $x$ in a possibly non-smooth manner. The functional associated to the above operator switches between $p$-growth on the set $\{a = 0\}$ and $q$-growth on the set $\{\alpha > 0\}$, behavior that warrants the development of novel techniques to investigate regularity. The first such set of results by Colombo and Mingione \cite{colombo2015regularity} describes - among other accomplishments - higher Lebesgue integrability of the functional's integrand under two fundamental assumptions: the H\"older continuity of the modulating coefficient $\alpha$ and the control of the ratio $q/p$ by a bound depending only on the dimension and the H\"older exponent of $\alpha$. 

Since the publication of \cite{colombo2015regularity} the theory of double-phase problems has been substantially expanded with connections to other areas; a comprehensive discussion is beyond the scope of this paper but we mention for instance \cite{pucci2018maximum, ok2018partial, byun2020calderon, byun2018riesz, de2018regularity, de2019regularity, byun2017global, baroni2018regularity, colombo2016calderon, baroni2015harnack, colombo2015bounded, DeFilippis-Mingione-Boarder} and the references they contain.
One such connection is to the regularity theory of fractional elliptic operators. The operator $\cL$ is the archetype of a class of nonlocal double phase operators first introduced in \cite{de2019holder}, in which the H\"older continuity of bounded viscosity solutions to $\cL u = f$ with bounded data $f$ was obtained. 
In this work we show regularity of solutions on a different scale; that under suitable assumptions on the data $f$, the modulating coefficient $a(\cdot,\cdot)$, and a certain ratio of integrability and differentiability exponents solutions $u$ to $\cL u = f$ exhibit a self-improvement property. Precisely, distributional solutions $u$ belonging to the fractional Sobolev space $W^{s,p}(\bbR^n)$ in fact belong to a Sobolev space with higher exponents of integrability and differentiability.

We assume that the modulating coefficient $a$ is measurable, and satisfies
\begin{equation}\label{Assumption:Intro:Coeff}
a(x,y) \in L^{\infty}(\bbR^{2n})\,, \qquad 0 \leq a(x,y) \leq M\,, \qquad a(x,y) = a(y,x)\,.
\tag{A1}
\end{equation}
In the case $a \equiv 0$ the operator $\cL$ reduces to the fractional $p$-Laplacian $(-\Delta)^s_p$. The regularity theory for the fractional $p$-Laplacian is quite extensive, and we refer the reader to \cite{schikorra2016nonlinear, brasco2018higher, brasco2017higher, palatucci2018dirichlet} and the references therein. One consequence of this article is the higher integrability of weak solutions to nonlocal degenerate elliptic equations of fractional $p$-Laplacian type with measurable coefficients; see Theorem \ref{thm:pLaplace}.

For this work we also require that 
\begin{equation}\label{Assumption:Intro:Exp}
p \leq q\,, \quad t \leq s\,, \qquad \frac{1}{p'} \leq \frac{tq}{sp} \leq 1\,, 
\tag{A2}
\end{equation}
where $p'$ is the H\"older conjugate of $p$: ${1\over p}+ {1\over p'}=1.$ 
Since solutions $u$ are assumed to only belong to $W^{s,p}$ the upper bound on $tq/sp$ therefore prevents the $tq$ term in the integrand of $\cL$ from becoming nonintegrable. The lower bound assumption effectively prevents the singularity in the integrand of $\cL$ from becoming too weak at infinity, so that the nonlocal tails can be controlled. This is in contrast to the local theory, in which the key constraint on the ratio $q/p$ is prescribed only from above. 
Additionally we will restrict ourselves to the case
\begin{equation}\label{Assumption:Intro:Exp2}
s p < n\,,
\tag{A3}
\end{equation}
because when $sp> n$ solutions will automatically belong to the H\"older class $C^{0,s-n/p}(\bbR^n)$ by Sobolev embedding for $sp > n$ regardless of the integrability conditions on the data $f$. See Remark \ref{rmk:ConditionsOnExponents} below for a further discussion on the natural character of these assumptions.

We aim to show higher differentiability and integrability of bounded solutions $u \in W^{s,p}(\bbR^n)$ to a weak formulation of the equation \eqref{eq:Intro:MainEqn}, that is
\begin{equation}\label{eq:Intro:WeakFormulation}
\cE(u,\varphi) = \intdm{\bbR^n}{f(x) \varphi(x)}{x}\,, \qquad \text{ for any } \varphi \in C^{\infty}_c(\bbR^n)\,,
\end{equation}
where the form $\cE(u,\varphi)$ is defined as
\begin{equation}\label{eq:Intro:WeakFormDefn}
\begin{split}
\cE(u,\varphi):= \iintdm{\bbR^n}{\bbR^n}{&\frac{|u(x)-u(y)|^{p-2}}{|x-y|^{n+sp}} (u(x)-u(y)) (\varphi(x)-\varphi(y)) \\
	&\quad + a(x,y) \frac{|u(x)-u(y)|^{q-2}}{|x-y|^{n+tq}} (u(x)-u(y))(\varphi(x)-\varphi(y))}{y}{x}\,.
\end{split}
\end{equation}
%
%
%
%
%
%
In order to prove such a result for solutions to \eqref{eq:Intro:WeakFormulation} we assume the data $f$ belongs to a Lebesgue space with sufficiently high exponent. We assume that for a given $\delta_0 > 0$
\begin{equation*}
f \in L^{p_{*_s}+\delta_0}_{loc}(\bbR^n)\,,
\end{equation*}
where we are using standard notation for H\"older and Sobolev exponents; that is, for any $r \in (1,\infty)$ and any $\sigma \in (0,1)$ we write
\begin{equation*}
r' = \frac{r}{r-1}\,, \qquad r^{*} = r^{*_{\sigma}} = \frac{n r}{n-\sigma r}\,, \qquad r_* = r_{*_{\sigma}} = \frac{n r'}{n+\sigma r'} = (r^*)'\,.
\end{equation*}
(The dependence of the embedding exponents on $\sigma$ will be suppressed whenever it is clear from context.) The integrability assumption on $f$ is a natural counterpart of the corresponding assumption necessary to prove higher integrability results for minimizers of energies associated to the local $p$-Laplacian; see for instance \cite[Chapter V, Section 3]{giaquinta1983multiple}.


Weak solutions $u$ are assumed to be \textit{a priori} bounded, a point clarified by the following definition:
\begin{definition}
A function $u \in W^{s,p}(\bbR^n) \cap L^{\infty}(\bbR^n)$ is a \textit{bounded weak solution} to \eqref{eq:Intro:MainEqn} with data $f$ if the nonlocal double phase energy $\cE(u,u) < \infty$ and if $u$ satisfies \eqref{eq:Intro:WeakFormulation}.
\end{definition}
We will show an ``intrinsic" higher differentiability and higher integrability for bounded weak solutions to \eqref{eq:Intro:MainEqn}. Precisely, if we denote the integrand of $\cE(u,u)$ by $P(x,y,u)$ so that \begin{equation*}
\cE(u,u) = \iintdm{\bbR^n}{\bbR^n}{P(x,y,u)}{y}{x}\,,
\end{equation*}
then by definition of $u$ as a bounded weak solution the function $P(\cdot,\cdot,u)$ belongs to $L^1(\bbR^{2n})$.
We are able to prove the following theorem concerning $P$, which constitutes the main result of this paper:

\begin{theorem}\label{thm:Intro:MainThm}
	Let $p$, $q \geq 2$ and $s$, $t \in (0,1)$ satisfy \eqref{Assumption:Intro:Exp}-\eqref{Assumption:Intro:Exp2} and let $a(x,y)$ satisfy \eqref{Assumption:Intro:Coeff}. Fix $\delta_0 > 0$, and let $f \in L^{p_{*_s} + \delta_0}_{loc}(\bbR^n)$. Let $u \in W^{s,p}(\bbR^n) \cap L^{\infty}(\bbR^n)$ be any bounded weak solution to \eqref{eq:Intro:MainEqn} with data $f$. Then there exists $\veps_0 \in (0,1)$ depending only on $n$, $p$, $q$, $s$, $t$, $M$, $\delta_0$ and $\Vnorm{u}_{L^{\infty}(\bbR^n)}$ such that for every $\tau \in (0,\veps_0)$
	\begin{equation*}
	P(\cdot,\cdot,u) \in L^{1+\tau}_{loc}(\bbR^{2n})\,.
	\end{equation*}
	In particular, there exist positive constants $\veps_1$ and $\veps_2$ such that $u \in W^{s+\veps_1, p+\veps_2}_{loc}(\bbR^n)$, and if $(s+\veps_1)(p+\veps_2) > n$ then $u$ is locally H\"older continuous.
\end{theorem}

 Explicit estimates on the constant $\veps_0$ can be obtained by tracing the dependencies through the proofs.
We work exclusively in the superquadratic case $p \in [2,\infty)$.
Note that Theorem \ref{thm:Intro:MainThm} does not treat the degenerate case $p$ and/or $q \in (1,2)$; this will be investigated in a future work.

\begin{remark}\label{rmk:ConditionsOnExponents}

Our assumption that $t \leq s$ can be thought of as imposing smoothness on the modulating coefficient. To see this, we write the integrand of $\cE(u,u)$ as
\begin{equation*}
\frac{|u(x)-u(y)|^p}{|x-y|^{n + s p}} + a(x,y)|x-y|^{(s-t)q} \frac{|u(x)-u(y)|^q}{|x-y|^{n + s q}}
\end{equation*}
The operator $\cL$ can therefore be read as the sum of a fractional $p$-Laplacian and a fractional $q$-Laplacian, both of differentiability order $s$, with the second operator perturbed by a coefficient $\wt{a} := a(x,y)|x-y|^{(s-t)q}$. Thus if $t < s$ then $\wt{a} \to 0$ as $|x-y| \to 0$. This ``uniform continuity" of $\wt{a}$ on the diagonal $x=y$ is in some sense a nonlocal analogue of the H\"older continuity condition on the modulating coefficient in local double-phase equations.
In fact, we can recast the upper bound in \eqref{Assumption:Intro:Exp} as
\begin{equation*}
q \leq p + \frac{(s-t)q}{s}\,.
\end{equation*}
In the context of proving regularity for \textit{a priori} bounded solutions, this is precisely the nonlocal analogue of the sharp condition $q \leq p + \beta$ in the local case, where $\beta$ is the H\"older continuity exponent of the modulating coefficent; see \cite{colombo2015bounded}.
\end{remark}

\subsection{Strategy of Proof}
To prove Theorem \ref{thm:Intro:MainThm} we use an argument developed by Kuusi, Mingione and Sire announced in \cite{kuusi2014fractional} and presented in \cite{kuusi2015} that builds a nonlocal fractional Gehring lemma in order to prove a self-improvement result for solutions to a class of monotone operators with quadratic growth related to the fractional Laplacian.
The arguments here are heavily based on the work and presentation done for the case $p=2$ in \cite{kuusi2015}. While it is apparent from a careful reading of that work that their methods apply to functionals with more general $p$-growth, the precise treatment of such classes of operators does not appear in the literature. Since we are further working with operators of mixed $(p,q)$ growth, we have included generalizations of the nonlocal reverse H\"older's inequality and fractional Gehring lemma that suit our context at the risk of repeating some arguments from \cite{kuusi2015}. We write the arguments of \cite{kuusi2015} for a general exponent $p$ so that the robustness  of their technique and as well as results can be clearly seen as applicable in a wealth of contexts. One such instances is, for example, this approach also extends to vector-valued solutions of nonlocal systems. A specific example is the strongly coupled system of nonlinear equations studied in \cite{Mengesha-Scott-Korn-in-Domain}. 

The paper \cite{kuusi2015} additionally considers the nonlocal generalization of the equation $\div (\bbA(\bx) \grad u) = \div \wt{f} + f$. We do not consider such a nonlocal divergence term on the right-hand side of the equation. Our results in the case $p = 2$ and $a \equiv 0$ therefore coincide with the results of \cite{kuusi2015} in the case when their data $g \equiv 0$.

Following the structure introduced in \cite{kuusi2015}, we define \textit{dual pairs} of measures and functions $(U,\nu)$. For small $\veps \in (0,1/p)$  we define the locally finite doubling Borel measure in $\bbR^{2n}$
\begin{equation}\label{eq:MeasureDefn}
\nu(A) := \int_{A} \frac{\rmd x \, \rmd y}{|x-y|^{n-\veps p}}\,, \qquad A \subset \bbR^{2n} \text{ measurable}\,,
\end{equation}
and we define the function
\begin{equation}
U(x,y) := \frac{|u(x)-u(y)|}{|x-y|^{s+\veps}}\,.
\end{equation}
It is then clear that 
\begin{equation*}
u \in W^{s,p}(\bbR^n) \qquad \text{ if and only if } \qquad U \in L^p(\bbR^{2n}; \nu)\,.
\end{equation*}
The integrand $P(x,y,u)$ of the energy $\cE(u,u)$ can be expressed in terms of $U$ as 
\begin{equation}\label{eq:UDefn:A}
[U^p + A(x,y) U^q]|x-y|^{-n+\epsilon p},\quad \text{where $A(x,y) := a(x,y) |x-y|^{(s-t)q + \veps(q-p)}$. }
\end{equation}
%
%
We can therefore write the double phase energy $\cE(u.u)$ in terms the dual pair as 
\begin{equation}\label{eq:UDefn:Energy}
\cE(u,u) = \intdm{\bbR^{2n}}{(U^p + A(x,y) U^q)}{\nu}=:\intdm{\bbR^{2n}}{G(x,y,U)}{\nu} 
\end{equation}
where the integrand $G(x,y,U):=U^p + A(x,y) U^q$. 
Then it now becomes clear that 
\begin{equation*}
P(\cdot,\cdot,u) \in L^1(\bbR^{2n}) \quad \text{ if and only if } \quad G(\cdot,\cdot,U) \in L^1(\bbR^{2n}; \nu)\,.
\end{equation*}

\begin{theorem}[Higher Regularity Result]\label{thm:Setup:MainRegularityResult}
	With all the assumptions of Theorem \ref{thm:Intro:MainThm}, there exists $\veps_0 > 0$ depending only on $\texttt{data}$ such that for every $\delta \in (0,\veps_0)$ we have
	\begin{equation}\label{eq:Setup:HigherRegForU}
	G(x,y,U) \in L^{1 + \delta}_{loc}(\bbR^{2d}; \nu)\,.
	\end{equation}
	where $\texttt{data} $ represents  $n,p,q,s,t,M,$ and $\Vnorm{u}_{L^{\infty}}$.
\end{theorem}

Theorem \ref{thm:Intro:MainThm} is a simple consequence of the above theorem. 
We will show \eqref{eq:Setup:HigherRegForU} directly,
and its proof relies on a reverse H\"older's inequality applied to the dual pair of function and measure $(G,\nu)$. The first step towards this is a suitable Caccioppoli-type inequality for $G$; see Theorem \ref{thm:CaccioppoliEstimate}. This inequality in turn relies on using the solution $u$ itself as an admissible test function, which is not so clear ahead of time, but possible to show that is indeed the case using an argument adapted from \cite{byun2017global} for our nonlocal context; see Theorem \ref{thm:AdmissibleTestFxn}. From this we derive a reverse H\"older inequality for $G = U^p + AU^q$ (Proposition \ref{prop:RevHolder}) involving a nonlocal tail; e.g.
\begin{equation*}
	\left( \fint_{B \times B} G \, \rmd \nu \right)^{1/p} \precsim \sum_{k=0}^{\infty} 
	\big( 2^{-k (\frac{sp}{p-1} -s - \veps)} +  2^{-k (\frac{tq}{p-1} -s - \veps)} \big) 
	\left( \fint_{2^k B \times 2^k B} U^{\eta} \, \rmd \nu \right)^{1/\eta} + \text{ term depending on } f
\end{equation*}
for every ball $B \subset \bbR^n$ and for some $\eta < p$.
This inequality holds only for diagonal sets of the type $B \times B \subset \bbR^{2n}$, and is insufficient to apply tools traditionally used to prove Gehring's lemma such as the maximal function. Nevertheless, Kuusi, Mingione, and Sire in \cite{kuusi2015} used a novel localization technique to show that the  reverse H\"older inequalities over diagonal ball is sufficient to prove a special fractional version of Gehring's lemma that is applicable for dual pairs of the above type. We will adapt this localization technique to our setting; see Section \ref{sec:GehringLemma}. Arguments with content very similar to that of \cite{kuusi2015} are left out of this work and presented in the companion note \cite{Scott-Mengesha-Levelset} for the sake of completeness.  We additionally refer to the original discussions and summaries of the technique in \cite{kuusi2015, kuusi2014fractional}.

\subsection{Consequences and Generalizations}

To streamline the presentation we present in this paper the proofs written only for the archetypal operator $\cL$. However, the real strength of these techniques become evident when considering a much wider class of operators. 
For example, define the form
\begin{equation*}
\begin{split}
\cE_{\phi_p, \phi_q, K_{sp}, K_{tq}}(u,\varphi) := \iintdm{\bbR^n}{\bbR^n}{& K_{sp}(x,y) \phi_p(u(x)-u(y)) (\varphi(x)-\varphi(y)) \\
	& \qquad   + a(x,y) K_{tq}(x,y) \phi_q(u(x)-u(y)) (\varphi(x)-\varphi(y))}{y}{x}\,,
\end{split}
\end{equation*}
where the kernels $K_{sp}$ and $K_{tq}$ are merely measurable and satisfy for ellipticity constants $\Lambda^-$ and $\Lambda^+$
%
%
\begin{equation}\label{eq:Intro:ConditionsOnGenKernels}
\begin{array}{l}
\Lambda^- |x-y|^{-n-sp} \leq K_{sp}(x,y) \leq  \Lambda^+ |x-y|^{-n-sp}\,, \\  \Lambda^- |x-y|^{-n-tq}  \leq K_{tq}(x,y) \leq \Lambda^+ |x-y|^{-n-tq}\,,
\end{array}\qquad
0 < \Lambda^- \leq \Lambda^+ < \infty\,.
\end{equation}
The measurable and monotone functions $\phi_p : \bbR \to \bbR$ and $\phi_q : \bbR \to \bbR$ satisfy
\begin{equation}\label{eq:Intro:ConditionsOnGenMonotons}
|\phi_r(z)| \leq \Lambda^+ |z|^{r-1}\,, \quad \phi_r(z)z \geq |z|^r\,, \qquad \text{ for all } z \in \bbR\,, \quad  r \in \{p, q \}\,.
\end{equation}
Then our results hold for solutions $u$ in $W^{s,p}(\bbR^n) \cap L^{\infty}(\bbR^n)$ to
\begin{equation}\label{eq:Intro:WeakFormGeneralOp}
\cE_{\phi_p, \phi_q, K_{sp}, K_{tq}}(u,\varphi) = \intdm{\bbR^n}{f(x) \varphi(x)}{x}\,, \text{for all $\phi\in C_c^{\infty}(\mathbb{R}^{n})$}. 
\end{equation}
To be precise we state the following theorem.
\begin{theorem}
Let $p$, $q \geq 2$ and $s$, $t \in (0,1)$ satisfy \eqref{Assumption:Intro:Exp}-\eqref{Assumption:Intro:Exp2} and let $a(x,y)$ satisfy \eqref{Assumption:Intro:Coeff}. Fix $\delta_0 > 0$, and let $f \in L^{p_* + \delta_0}_{loc}(\bbR^n)$. Let $u \in W^{s,p}(\bbR^n) \cap L^{\infty}(\bbR^n)$ be any bounded weak solution to \eqref{eq:Intro:WeakFormGeneralOp} with data $f$. Then there exists $\veps_0 \in (0,1)$ depending only on the ellipticity constants $\Lambda^+$ and $\Lambda^-$, $n$, $p$, $q$, $s$, $t$, $M$, $\delta_0$ and $\Vnorm{u}_{L^{\infty}(\bbR^n)}$ such that for every $\delta \in (0,\veps_0)$
\begin{equation*}
P_{\phi,K}(\cdot,\cdot,u) \in L^{1+\delta}_{loc}(\bbR^{2n})\,,
\end{equation*}
where $P_{\phi,K}(x,y,u)$ is the integrand of $\cE_{\phi_p, \phi_q, K_{sp}, K_{tq}}(u,u)$. In particular, there exist positive constants $\veps_1$ and $\veps_2$ such that $u \in W^{s+\veps_1, p+\veps_2}_{loc}(\bbR^n)$.
\end{theorem}
%

A notable special case is when $a \equiv 0$. In this situation we obtain regularity results for a wide class of operators related to the fractional $p$-Laplacian. Upon careful inspection of the forthcoming proofs one should note that if $a \equiv 0$ then solutions need not be bounded, and we have the following theorem as a consequence:

\begin{theorem}\label{thm:pLaplace}
Let $p \geq 2$ and $s \in (0,1)$ satisfy \eqref{Assumption:Intro:Exp2}. Fix $\delta_0 > 0$, and let $f \in L^{p_* + \delta_0}_{loc}(\bbR^n)$. Let $K_{sp}$ satisfy \eqref{eq:Intro:ConditionsOnGenKernels} and $\phi_p$ satisfy \eqref{eq:Intro:ConditionsOnGenMonotons}. Suppose that $u \in W^{s,p}(\bbR^n)$ satisfies
\begin{equation*}
 \iintdm{\bbR^n}{\bbR^n}{K_{sp}(x,y) \phi_p(u(x)-u(y)) (\varphi(x)-\varphi(y))}{y}{x} = \intdm{\bbR^n}{f(x) \varphi(x)}{x}
\end{equation*}
for every $\varphi \in C^{\infty}_c(\bbR^n)$.
Then there exists $\veps_0 \in (0,1)$ depending only on the ellipticity constants $\Lambda^+$ and $\Lambda^-$, $n$, $p$, $s$, and $\delta_0$ but not $u$ such that for every $\delta \in (0,\veps_0)$
\begin{equation*}
u \in W^{s+\delta, p+\delta}_{loc}(\bbR^n)\,.
\end{equation*}
\end{theorem}
Interior Sobolev regularity for solutions to the Poisson problem for the fractional $p$-Laplacian (that is, for the above operator with $\phi_p(t) = |t|^{p-2}t$ and $K_{s,p}(x,y) = |x-y|^{-n-sp}$) is proven in \cite{brasco2017higher}. An \textit{a priori} estimate in the spirit of Theorem \ref{thm:pLaplace} for smooth solutions to ``regional" operators of the above type can be found in \cite{schikorra2016nonlinear}.


\subsection{A Fractional Gehring Lemma for General Sobolev Functions}
%
%
%
%
%

We state here a version of the Fractional Gehring Lemma valid for general Sobolev functions. For $p=2$ this is exactly the statement of \cite[Theorem 1.3]{kuusi2015}. 

\begin{theorem}
	Suppose $u \in W^{s,p}(\bbR^n)$ for $s \in (0,1)$, $p \geq 2$. Let $\eta \in [1,p)$ be fixed, let $\veps \in (0,s/p)$, let $\{\alpha_k\} \in \ell^1$ and let $(U,\nu)$ be the dual pair generated by $u$. Suppose the following reverse H\"older-type inequality holds for any $\sigma \in (0,1)$ and for any ball $B \subset \bbR^n$ and $\mathcal{B} = B\times B$:
\begin{equation*}
\begin{split}
\left( \fint_{\frac{1}{4} \cB} U^{p} \, \rmd \nu \right)^{1/p} &\leq \frac{C}{\sigma \veps^{1/\eta - 1/p}} \left( \fint_{\cB} U^{\eta} \, \rmd \nu \right)^{1/\eta} + \frac{C \sigma}{ \veps^{1/\eta - 1/p}} \sum_{k=0}^{\infty} \alpha_k \left( \fint_{2^k \cB} U^{\eta} \, \rmd \nu \right)^{1/\eta}\,.
\end{split}
\end{equation*}
Then there exists a $\delta_0 >0$ depending only on $n$, $s$, $p$, $\eta$, $\sigma$, $\alpha_k$ and $\veps$ such that for all $\delta \in (0,\delta_0)$ the function $u \in W^{s+\delta,p+\delta}_{loc}(\bbR^n)$, with the following inequality holding for a constant $C$ depending on  $n$, $s$, $p$, $\eta$, $\sigma$, $\alpha_k$ and $\veps$:
\begin{equation*}
	\begin{split}
		\left( \fint_{\frac{1}{4} \cB} U^{p+\delta} \, \rmd \nu \right)^{1/(p+\delta)} &\leq C \sum_{k=0}^{\infty} \alpha_k \left( \fint_{2^k \cB} U^{p} \, \rmd \nu \right)^{1/p}\,.
	\end{split}
\end{equation*}
\end{theorem}

This paper is organized as follows: In the next section we identify notation and conventions, and show that bounded weak solutions to \eqref{eq:Intro:MainEqn} can be used as test functions in the weak formulation. The Caccioppoli inequality is proved in Section \ref{sec:Caccioppoli}, and the reverse H\"older inequality is proved in Section \ref{sec:RevHolder}. In Section \ref{sec:GehringLemma} we establish the fractional Gehring lemma and associated higher differentiability of solutions. The Gehring lemma relies on an estimate of the level sets of $G$; its proof is quite technical but the argument used very closely resembles that of the corresponding result for $p=2$ found in \cite{kuusi2015}. For completeness, we have placed its proof in the companion note \cite{Scott-Mengesha-Levelset}.

\section{Preliminaries}


Throughout, we denote positive constants by $c$, $C$, etc., and they may change from line to line. We list the dependencies in parentheses after the constant when we wish to make them explicit, i.e.\ if a constant $C$ depends only on $n$, $p$ and $s$, we write $C = C(n,p,s)$. We will abbreviate the following set of parameters as
\begin{equation*}
\texttt{data} \equiv (n,p,q,s,t,M,\Vnorm{u}_{L^{\infty}})\,.
\end{equation*}
In $\bbR^n$, denote the open ball of radius $R$ centered at $x_0$ by
\begin{equation*}
B(x_0,R) = B_R(x_0) := \{ x \in \bbR^n \, : \, |x-x_0| < R \}\,.
\end{equation*}
We will sometimes denote the ball $B \equiv B_R \equiv B_R(x_0)$ whenever the center and/or radius is clear from context. If $B$ is a ball centered at $x_0$ with radius $R$, then $\sigma B$ is the ball centered at $x_0$ with radius $\sigma R$.
Given any measure $\mu$, denote the average of a $\mu$-measurable function $h$ over a set $\cA$ by
\begin{equation*}
(h)_{\cA} := \fint_{\cA} h \, \rmd \mu = \frac{1}{\mu(\cA)} \intdm{\cA}{h(x)}{\mu}\,.
\end{equation*}
In dealing with functions defined on $\bbR^{2n}$ such as $U$, we consider the norm on $\bbR^{2n}$ defined by
\begin{equation*}
\Vnorm{(x,y)} := \max \{ |x|, |y| \}\,,
\end{equation*}
where $|\cdot|$ denotes the Euclidean norm on $\bbR^n$.
Denote the balls defined by this norm as
\begin{equation*}
\begin{split}
\cB(x_0,y_0,R) &:= \{ (x,y) \in \bbR^n \times \bbR^n \, : \, \Vnorm{(x,y)-(x_0,y_0)} < R\} \\
&= B(x_0,R) \times B(y_0,R)\,.
\end{split}
\end{equation*}
If we denote
\begin{equation*}
B_{\bbR^{2n}}(x_0,y_0,R) := \{ (x,y) \in \bbR^n \times \bbR^n \, : \, \sqrt{|x-x_0|^2 + |y-y_0|^2} < R \}\,,
\end{equation*}
then clearly
\begin{equation*}
B_{\bbR^{2n}}(x_0,y_0,R) \subset \cB(x_0,y_0,R) \subset B_{\bbR^{2n}}(x_0,y_0,2R)\,. 
\end{equation*}
Often we will need to consider balls in $\bbR^{2n}$ centered at a point on the ``diagonal," that is, a point of the form $(x_0,x_0)$ for $x_0 \in \bbR^n$. In this case we abbreviate $\cB(x_0,x_0,R) \equiv \cB(x_0,R)$. We will also use the abbreviations $\cB(x_0,R) \equiv \cB_R(x_0) \equiv \cB_R \equiv \cB$ whenever the center and/or radius is clear from context. Whenever there is no ambiguity we write $\cB(x_0,\sigma R) = \sigma \cB$.
We also denote
\begin{equation*}
\text{Diag} := \{ (x,x) \, : \, x \in \bbR^n \}\,.
\end{equation*}
We will use the elementary inequality
\begin{equation}\label{eq:GeometricSeriesEstimate}
2^{k r} \sum_{j=k-1}^{\infty} 2^{-j r} \leq \frac{4^{ r}}{r \ln(2)}\,, \qquad \text{ for } k \geq 1 \quad \text{and} \quad r \in (0,\infty)\,.
\end{equation}
The cardinality of a finite set $\cA$ is denoted by $\# \cA$. The set of nonnegative integers $\{ 0,1,2, \ldots \}$ is designated by $\bbZ_+$.

For any domain $\Omega \subset \bbR^n$, $0 < \sigma <1$ and $r \in [1,\infty)$ the fractional Sobolev spaces are defined by the Gagliardo seminorm
\begin{equation*}
W^{\sigma,r}(\Omega) := \left\{ u \in L^r(\Omega) \, : \, [u]_{W^{\sigma,r}(\Omega)} :=  \iintdm{\Omega}{\Omega}{\frac{|u(x)-u(y)|^r}{|x-y|^{n+\sigma r}} }{y}{x} < \infty \right\}
\end{equation*}
with norm $\Vnorm{\cdot}_{W^{\sigma,r}(\Omega)}^r :=  \Vnorm{\cdot}_{L^{r}(\Omega)}^r + [\cdot]_{W^{\sigma,r}(\Omega)}^r$.

We will also use the following fractional Poincar\'e-Sobolev-type inequalities throughout the paper. A proof of the first can be found in several places; see for instance \cite{DNPV12, bass2005holder}. The second can be found in \cite{mingione2003singular, schikorra2016nonlinear}.

\begin{theorem}[Fractional Poincar\'e-Sobolev Inequality]\label{thm:SobolevInequality}
	Let $r \in [1,\infty)$, $0<\sigma<1$. Let $B = B_R(x_0)$ for some $R>0$, $x_0 \in \bbR^n$. Then there exists $C = C(n,r,\sigma) > 0$ such that 
	\begin{equation*}
		\left( \fint_B \left| \frac{v(x) - (v)_B}{R^{\sigma}} \right|^{r^{*_{\sigma}}} \, \rmd x \right)^{1/{r^{*_{\sigma}}}} \leq C \left( \int_B \fint_B \frac{|v(x)-v(y)|^r}{|x-y|^{n+\sigma r}} \, \rmd y \, \rmd x \right)^{1/r}
	\end{equation*}
	for every $v \in W^{\sigma,r}(B)$.
\end{theorem}

\begin{theorem}[Fractional Poincar\'e Inequality]\label{thm:PoincareInequality}
	Let $r \in [1,\infty)$, $0<\sigma<1$. Let $B = B_R(x_0)$ for some $R>0$, $x_0 \in \bbR^n$. Then there exists $C = C(n,r) > 0$ such that 
	\begin{equation*}
		\left( \fint_B \left| \frac{v(x) - (v)_B}{R^{\sigma}} \right|^{r} \, \rmd x \right)^{1/r} \leq C \left( \int_B \fint_B \frac{|v(x)-v(y)|^r}{|x-y|^{n+\sigma r}} \, \rmd y \, \rmd x \right)^{1/r}
	\end{equation*}
	for every $v \in W^{\sigma,r}(B)$.
\end{theorem}

\subsection{Admissible Test Functions}

\begin{theorem}\label{thm:AdmissibleTestFxn}
Let $u \in W^{s,p}(\bbR^n) \cap L^{\infty}(\bbR^n)$ satisfy \eqref{eq:Intro:WeakFormulation} with data $f \in L^{p_* + \delta_0}_{loc}(\bbR^n)$. Let $B = B(x_0,r) \subset \bbR^n$ be an arbitrary ball. Then every $w \in W^{s,p}(\bbR^n) \cap L^{\infty}(\bbR^n)$ with $\cE(w,w) < \infty$ and $\supp w \subset \frac{1}{2}B$ satisfies
\begin{equation}\label{eq:AdmissibleTestFxn:MainResult}
\cE(u,w) = \intdm{\bbR^n}{f(x) w(x)}{x}\,.
\end{equation}
\end{theorem}

\begin{proof}
	It suffices to prove \eqref{eq:AdmissibleTestFxn:MainResult} for $B = B_1(0)$; the general case will follow by a scaling argument. Indeed, for any $R>0$ and $x_0 \in \bbR^n$ and for any $w \in  W^{s,p}(\bbR^n) \cap L^{\infty}(\bbR^n)$  with $\cE(w,w) < \infty$ and $\supp w \subset \frac{1}{2}B(x_0,R)$ define the functions
	\begin{equation*}
	\begin{split}
	v(x) = u(x_0 + R x) \in W^{s,p}(\bbR^n) \cap L^{\infty}(\bbR^n)\,, \\
	\wt{w}(x) = w(x_0 + Rx) \in  W^{s,p}(\bbR^n) \cap L^{\infty}(\bbR^n)\,, \\
	\widetilde{f}(x) = R^{sp} f(x_0+Rx) \in L^{p_*+\delta_0}_{loc}(\bbR^n)\,.
	\end{split}
	\end{equation*}
Then $\cE(\wt{w},\wt{w}) < \infty$ and $\supp \wt{w} \subset B(0,1/2)$. An application of \eqref{eq:AdmissibleTestFxn:MainResult} for $R=1$ and $x_0 = 0$ then gives
\begin{equation*}
\wt{\cE}(v,\wt{w}) = \intdm{\bbR^n}{\wt{f}(x) \wt{w}(x)}{x}\,,
\end{equation*}
where
	\begin{align*}
	\widetilde{\cE}(v,\wt{w}) &:= \iintdm{\bbR^n}{\bbR^n}{\frac{|v(x)-v(y)|^{p-2}}{|x-y|^{n+sp}} (v(x)-v(y))(\wt{w}(x)-\wt{w}(y)) \\
	&\qquad\quad+ \widetilde{a}(x,y) \frac{|v(x)-v(y)|^{q-2}}{|x-y|^{n+tq}} (v(x)-v(y))(\wt{w}(x)-\wt{w}(y))}{y}{x}
	\end{align*}
and
	\begin{equation*}
	\widetilde{a}(x,y) = R^{sp-tq} a(x_0+Rx,x_0+Ry)\,.
	\end{equation*}
Note that the function $\wt{a}$ satisfies \eqref{Assumption:Intro:Coeff}. Therefore \eqref{eq:AdmissibleTestFxn:MainResult} for general $x_0$ and $R$ follows by rescaling.

We will first show that for any $w \in W^{s,p}(\bbR^n) \cap L^{\infty}(\bbR^n)$ with $\cE(w,w) < \infty$ and $\supp w \subset \frac{1}{2}B$ there exists a sequence $\{w_j \} \subset C^{\infty}_c(\frac{3}{2}B)$ (regarded as defined on all of $\bbR^n$ via extension by zero) such that
\begin{equation}\label{eq:AdmissibleTestFxn:Proof0.5}
w_j \to w \quad \text{ in } L^p(\bbR^n)\,, \qquad P(x,y,w_j) \to P(x,y,w) \quad \text{ in } L^1(\bbR^{2n})\,.
\end{equation}
Let $\psi \in C^{\infty}_c(B_1(0))$ be a standard mollifier with $\psi \geq 0$, $\Vnorm{\psi}_{L^1(\bbR^n)}=1$, and define $\psi_{\tau}(x) := \frac{1}{\tau^n} \psi \left( \frac{x}{\tau} \right)$ for $x \in B(0,\tau)$ with $\tau > 0$. For $0 < \tau < 1/4$ define $w_{\tau} := w \ast \psi_{\tau} \in C^{\infty}_c(B_{1+\tau})$. We claim that we can choose a subsequence $\tau_i$ such that $\{w_{\tau_j}\}$ satisfies \eqref{eq:AdmissibleTestFxn:Proof0.5}. To that end, notice first that since $w \in L^{\infty}(\bbR^n)$ we have $\lim_{\tau \to 0} \Vnorm{w_{\tau} - w}_{L^m(\sigma B)} =0$ for all $m \in [1,\infty)$ and for all $\sigma > 0$, and thus there exists a subsequence (not relabeled) such that $P(x,y,w_{\tau}) \to P(x,y,w)$ almost everywhere in $\bbR^{2n}$. In fact, we will show that $P(x,y,w_{\tau}) \to P(x,y,w)$ in $L^1(\bbR^{2n})$ as $\tau \to 0$. 
Directly, $P(x,y,w_{\tau}) \to P(x,y,w)$ in $L^1(\bbR^{2n} \setminus \cB(0,2))$:
\begin{equation*}
\begin{split}
\intdm{\bbR^{2n} \setminus \cB(0,2)}{& |P(x,y,w_{\tau}) -P(x,y,w)| }{y}\, \rmd x \\
&= 2 \iintdm{B_{3/2}(0)}{\bbR^n \setminus B_2(0)}{ \left|  \frac{|w_{\tau}(x)|^p - |w(x)|^p}{|x-y|^{n+sp}} +a(x,y) \frac{|w_{\tau}(x)|^q-|w(x)|^q}{|x-y|^{n+tq}} \right| }{y}{x} \\
&\leq C \Vnorm{w_{\tau} - w}_{L^p(B_{3/2}(0))} + C \Vnorm{w_{\tau} - w}_{L^q(B_{3/2}(0))} \rarrowop_{\tau \to 0} 0\,
\end{split}
\end{equation*}
where in the last line we used the algebraic inequality $||a|^{p} - |b|^{p}| \leq p|a-b| (|a|^{p-1} + |b|^{p-1})$ which holds true for any $p\geq 2$ and $a, b\in \mathbb{R}$ followed by H\"older's inequality. 

Thus it remains to show that $P(x,y,w_{\tau}) \to P(x,y,w)$ in $L^1(\cB(0,2))$ as $\tau \to 0$. We will show that there exist $\bbR^{2n}$-integrable functions $\Phi_{\tau}$ such that $P(x,y,w_{\tau}) \leq \Phi_{\tau}(x,y)$ pointwise in $\cB(0,2)$ for all $\tau \in (0,1/4)$ and that  $\Phi_{\tau} \to P(\cdot,\cdot,w)$ in $L^1(\bbR^{2n})$. This will imply convergence of $P(x,y,w_{\tau})$ in $L^1(\cB(0,2))$ by the Generalized Dominated Convergence Theorem and thus \eqref{eq:AdmissibleTestFxn:Proof0.5} will be proved. 
To find such a function $\Phi_{\tau}$, we introduce the expressions 
\begin{equation*}
\begin{split}
 a_{\tau}(x,y) := \inf_{z \in B_{\tau}(0) } a(x-z,y-z)\,, \text{and}\\
	P_{\tau}(x,y,v) := \frac{|v(x)-v(y)|^p}{|x-y|^{n+sp}} + a_{\tau}(x,y) \frac{|v(x)-v(y)|^q}{|x-y|^{n+tq}}\,.
\end{split}
\end{equation*}
For $x$ and $y$ in $B_2(0)$, since $sp - tq \geq 0$ we have 
\begin{equation*}
\begin{split}
|w_{\tau}(x)-w_{\tau}(y)|^{q-p} |x-y|^{sp-tq} &\leq 4^{sp-tq} \left| \intdm{\bbR^n}{|\psi_{\tau}(x-z) - \psi_{\tau}(y-z)| |w(z)|}{z} \right|^{q-p} \\
&\leq C \Vnorm{\psi}_{L^{1}(B)}^{q-p} \Vnorm{w}_{L^{\infty}(\bbR^n)}^{q-p} \leq \wt{C}\,,
\end{split}
\end{equation*}
where $\wt{C}$ is independent of $\tau$. Then
\begin{equation*}
\begin{split}
P(x,y,w_{\tau}) &\leq |a(x,y)-a_{\tau}(x,y)| \frac{|w_{\tau}(x)-w_{\tau}(y)|^q}{|x-y|^{n+tq}} + P_{\tau}(x,y,w_{\tau}) \\
	&\leq 2M \wt{C}  \frac{|w_{\tau}(x)-w_{\tau}(y)|^p}{|x-y|^{n+sp}} + P_{\tau}(x,y,w_{\tau}) \leq C P_{\tau}(x,y,w_{\tau})\,.
\end{split}
\end{equation*}
To further estimate $P_{\tau}(x,y,w_{\tau})$ we see that from the definition of $a_{\tau}$
\begin{equation*}
\begin{split}
P_{\tau}(x,y,w_{\tau}) &\leq \intdm{B_{\tau}(0)}{ \frac{|w(x-z)-w(y-z)|^p}{|x-y|^{n+sp}} \psi_{\tau}(z) +a_{\tau}(x,y) \frac{|w(x-z)-w(y-z)|^q}{|x-y|^{n+tq}}  \psi_{\tau}(z) }{z} \\
	&\leq \intdm{B_{\tau}(0)}{ \frac{|w(x-z)-w(y-z)|^p}{|x-y|^{n+sp}} \psi_{\tau}(z) +a(x-z,y-z) \frac{|w(x-z)-w(y-z)|^q}{|x-y|^{n+tq}}  \psi_{\tau}(z) }{z} \\
	&= \intdm{B_{\tau}(0)}{ P(x-z,y-z,w)  \psi_{\tau}(z) }{z}\,.
\end{split}
\end{equation*}Take $\Phi_{\tau}(x, y) = \intdm{B_{\tau}(0)}{ P(x-z,y-z,w)  \psi_{\tau}(z) }{z}$. Then we have from the above calculation that $P(x,y,w_{\tau}) \leq C \Phi_{\tau}(x, y)$ 
 for all $x$ and $y$ in $B_2(0)$. Further, $\Phi_{\tau}(x, y) \to P(x,y,w)$ in $L^1(\bbR^{2n})$. Indeed, 
setting  
\begin{equation*}
\wt{x} = (x,y)\,, \qquad \wt{z} := (z,z)\,, \qquad V(\wt{x}) = v(x,y) = P(x,y,w)\,,
\end{equation*}
we have that, after change of variables and interchanging integrals, 
\begin{align*}
\int_{\mathbb{R}^{2n}} |\Phi_{\tau}(x, y) -P(x,y,w)|\, \rmd x \, \rmd y &\leq \intdm{\bbR^{2n}}{\intdm{B_{\tau}(0)}{ \psi_{\tau}(z) |V(\wt{x}) - V(\wt{x} - \wt{z}) | }{z}}{\wt{x}}\\
&=\intdm{B_1(0)}{  \psi(z) \intdm{\bbR^{2n}}{ |V(\wt{x}) - V(\wt{x} - \tau \wt{z}) | }{\wt{x}}}{z} \,.
\end{align*}
We claim that the latter converges to $0$ as $\tau\to 0$. To see this,  $V$ belongs to $L^{1}(\mathbb{R}^{2n})$ by assumption, and so $\lim\limits_{\tau \to 0} \intdm{\bbR^{2n}}{ |V(\wt{x}) - V(\wt{x} - \tau \wt{z}) | }{\wt{x}} = 0$ for each $z\in B_1(0)$ by continuity of translations in $L^1(\bbR^{2n})$.  Moreover, 
\[
 \psi(z)  \intdm{\bbR^{2n}}{ |V(\wt{x}) - V(\wt{x} - \tau \wt{z}) | }{\wt{x}} \leq 2\psi(z)\|v\|_{L^{1}}. 
\]
The result now follows by the dominated convergence theorem. 

Next we show that
\begin{equation}\label{eq:AdmissibleTestFxn:Proof2}
\cE(u,w_j) \to \cE(u,w) \qquad \text{ as } j \to \infty\,.
\end{equation}
We denote the integrand of $\cE(u,w_j)$ as $T(x,y,u,w_j)$.
By Young's inequality
\begin{equation*}
\begin{split}
|T(x,y,u,w_j)| &\leq \left( \frac{|u(x)-u(y)|^{p-1}}{|x-y|^{n+sp}} + {a}(x,y) \frac{|u(x)-u(y)|^{q-1}}{|x-y|^{n+tq}} \right) |w_j(x)-w_j(y)| \\
	&\leq C (P(x,y,u) + P(x,y,w_j))\,.
\end{split}
\end{equation*}
Thus \eqref{eq:AdmissibleTestFxn:Proof2} follows from \eqref{eq:AdmissibleTestFxn:Proof0.5} and the Generalized Dominated Convergence Theorem. Finally 
we have already noted that $w_j \to w$ in $L^{p^*}(B)$, and since $f \in L^{p_*}(B)$ we have $\intdm{\bbR^n}{f w_j}{x} \to \intdm{\bbR^n}{f w}{x}$. The proof is complete.
\end{proof}

\begin{remark}\label{remarka=0}
	If $a \equiv 0$ then the proof of Theorem \ref{thm:AdmissibleTestFxn} is much easier. Indeed, any $w \in W^{s,p}(\bbR^n)$ with compact support (say contained in a ball $B$) is an admissible test function. If $w_j$ is a sequence in $C^{\infty}_c(\bbR^n)$ converging to $w$ in $W^{s,p}(\bbR^n)$, then by H\"older's inequality
	\begin{equation*}
	\begin{split}
	|\cE(u,w_j) - \cE(u,w)| &\leq \iintdm{\bbR^n}{\bbR^n}{\frac{|u(x)-u(y)|^{p-1}}{|x-y|^{n+sp}} |w_j(x)-w_j(y)-(w(x)-w(y))|}{y}{x} \\ &
	\leq [u]_{W^{s,p}(\bbR^n)}^{p-1} [w_j-w]_{W^{s,p}(\bbR^n)}\,,
	\end{split}
	\end{equation*}
	which converges to zero as $j \to \infty$. Then $w_j \to w$ in $L^{p^*}(B)$, and since $f \in L^{p_*}(B)$ we have $\textstyle \intdm{\bbR^n}{f w_j}{x} \to \intdm{\bbR^n}{f w}{x}$, and the proof of Theorem \ref{thm:AdmissibleTestFxn} in the case $a=0$ is finished.
%
\end{remark}

\section{The Caccioppoli Inequality}\label{sec:Caccioppoli}

\begin{theorem}\label{thm:CaccioppoliEstimate}
Let $u \in  W^{s,p}(\bbR^n) \cap L^{\infty}(\bbR^n)$ be a weak solution to \eqref{eq:Intro:MainEqn}. Let $B = B_R(x_0) \subset \bbR^n$ be a ball, and let $ \varphi \in C^{\infty}_c( B)$ such that $0 \leq \varphi \leq 1$, $\supp \varphi \subset \frac{1}{2} B$ and $|\grad \varphi| \leq \frac{C(n)}{R}$. Then for some $C = C(\texttt{data}) > 0$ we have 
\begin{equation}\label{eq:CaccioppoliEstimate}
\begin{split}
& \iintdm{B}{B}{ \frac{|\varphi^{q/p}(x) u(x) - \varphi^{q/p}(y) u(y)|^p}{|x-y|^{n+sp}} }{y}{x} 
+  \iintdm{B}{B}{a(x,y) \frac{|\varphi(x) u(x) - \varphi(y) u(y)|^q}{|x-y|^{n+tq}} }{y}{x} \\
	&\leq \frac{C}{R^{sp}} \int_B |u(x)|^p \, \rmd x + C \iintdm{B}{B}{a(x,y) \frac{|\varphi(x) - \varphi(y)|^q}{|x-y|^{n+tq}} |u(x)|^q }{y}{x}\\
	&+ C \intdm{B}{\varphi^q(x) |u(x)| }{x} \intdm{\bbR^n \setminus B}{\frac{|u(y)|^{p-1}}{|x_0-y|^{n+sp}}}{y} \\
	&+C \iintdm{B}{\bbR^n \setminus B}{a(x,y) \varphi^q(x) \frac{|u(x)|^{q} + |u(y)|^{q-1}|u(x)|}{|x_0-y|^{n+tq}}}{y}{x} + C R^{n+sp'} \left( \fint_B |f(x)|^{p_*} \, \rmd x \right)^{p'/p_*}.  
\end{split}
\end{equation}
\end{theorem}

\begin{proof}
Following standard approaches, we take $\varphi^q u$ as the test function in \eqref{eq:Intro:WeakFormulation}. We can make this choice by using Theorem \ref{thm:AdmissibleTestFxn}.
Indeed, clearly $\varphi^q u \in W^{s,p}(\bbR^n) \cap  L^{\infty}(\bbR^n)$, $\supp \varphi^q u \subset \frac{1}{2}B$. Moreover,   
\begin{equation*}
\begin{split}
\cE(\varphi^q u,\varphi^q u) &\leq C \iintdm{\bbR^n}{\bbR^n}{\frac{|\varphi^q(x) - \varphi^q(y)|^p}{|x-y|^{n+sp}} |u(y)|^p }{y}{x} + C \iintdm{\bbR^n}{\bbR^n}{\frac{|u(x) - u(y)|^p}{|x-y|^{n+sp}} |\varphi(x)|^p }{y}{x} \\
	&\quad + C\iintdm{\bbR^n}{\bbR^n}{a(x,y) \frac{|\varphi^q(x) - \varphi^q(y)|^q}{|x-y|^{n+tq}} |u(y)|^q }{y}{x}\\
	& \quad\quad+ C \iintdm{\bbR^n}{\bbR^n}{a(x,y) \frac{|u(x) - u(y)|^q}{|x-y|^{n+tq}} |\varphi(x)|^q }{y}{x} \\
	&\leq C \cE(u,u) + C \max \{ \Vnorm{u}_{L^{\infty}}^p, \Vnorm{u}_{L^{\infty}}^q \} \cdot \cE(\varphi^q,\varphi^q) < \infty\,.
\end{split}
\end{equation*}
Using $\varphi^q u$ in the definition \eqref{eq:Intro:WeakFormulation} we have as in \eqref{eq:AdmissibleTestFxn:MainResult} that $\cE(u,\varphi^q u) = \intdm{B}{\varphi^q(x) f(x) \cdot u(x)}{x}$. 
Writing $\cE(u,\varphi^q u) = \rmI+ {\rm II}$ where 
\begin{equation*}
\begin{split}
{\rm I} &= \iintdm{B}{B}{\left[\frac{|u(x)-u(y)|^{p-2}}{|x-y|^{n+sp}} (u(x)-u(y)) (\varphi^q(x) u(x)-\varphi^q(y)u(y))\right. \\
		&\qquad + \left.a(x,y) \frac{|u(x)-u(y)|^{q-2}}{|x-y|^{n+tq}} (u(x)-u(y))(\varphi^q(x) u(x)-\varphi^q(y)u(y))\right]}{y}{x} \\
	{\rm II}&= 2\iintdm{B}{\bbR^n \setminus B}{\left[\frac{|u(x)-u(y)|^{p-2}}{|x-y|^{n+sp}} (u(x)-u(y)) \varphi^q(x) u(x)\right. \\
		&\qquad + \left.a(x,y) \frac{|u(x)-u(y)|^{q-2}}{|x-y|^{n+tq}} (u(x)-u(y)) \varphi^q(x) u(x)\right]}{y}{x}
\end{split}
\end{equation*}
we will estimate each integral separately, then collect terms. 

\noindent{\bf Estimate of $\rmI$.} Write 
\begin{equation*}
\begin{split}
\rmI = \iintdm{B}{B}{& \frac{|u(x)-u(y)|^{p-2}}{|x-y|^{n+sp}} (u(x)-u(y)) (\varphi^q(x) u(x)-\varphi^q(y)u(y))}{y}{x} \\
	&\quad + \iintdm{B}{B}{ a(x,y) \frac{|u(x)-u(y)|^{q-2}}{|x-y|^{n+tq}} (u(x)-u(y))(\varphi^q(x) u(x)-\varphi^q(y)u(y))}{y}{x} := \rmI_1 + \rmI_2\,.
\end{split}
\end{equation*}
We will estimate $\rmI_1$ first, and a similar estimate will hold for $\rmI_2$. 

We assume first that $\varphi(x) \geq \varphi(y)$. By adding and subtracting $\varphi^q(x) u(y)$,
\begin{equation}\label{eq:Caccioppoli:IEstimate1}
\begin{split}
|u(x)-u(y)|^{p-2} & (u(x)-u(y)) (\varphi^q(x)u(x) - \varphi^q(y) u(y)) \\
	&= \varphi^q(x) |u(x) - u(y)|^p
	+ \big( \varphi^q(x) - \varphi^q(y) \big) |u(x)-u(y)|^{p-2} (u(x)-u(y)) u(y) \\
	&= \varphi^q(x) |u(x) - u(y)|^p + \rmR_1\,.
\end{split}
\end{equation}
We will bound $\rmR_1$ from below. Set 
\begin{equation*}
\widetilde{\varphi}(x) := \varphi^{q/p}(x)\,.
\end{equation*}
Then $|\widetilde{\varphi}(x)| \leq |\varphi(x)|$ since $0 \leq \varphi \leq 1$ and $q/p \geq 1$, and
\begin{equation*}
| \grad[ \widetilde{\varphi}(x) ]| = \left| \frac{q}{p} \varphi^{q/p - 1}(x) \grad \varphi(x) \right| \leq \frac{q}{p} |\grad \varphi(x)|\,.
\end{equation*}
Now, by the assumption $\varphi(x) \geq \varphi(y)$ we have $\widetilde{\varphi}(x) \geq \widetilde{\varphi}(y)$, and so
\begin{equation}\label{eq:Caccioppoli:IEstimate2}
\begin{split}
\varphi^q(x) - \varphi^q(y) &= p \big( \sigma \widetilde{\varphi}(x) + (1-\sigma) \widetilde{\varphi}(y) \big)^{p-1} (\widetilde{\varphi}(x) - \widetilde{\varphi}(y)) \\
	&\geq - p \big| \sigma \widetilde{\varphi}(x) + (1-\sigma) \widetilde{\varphi}(y) \big|^{p-1} |\widetilde{\varphi}(x) - \widetilde{\varphi}(y)| \geq - p |\widetilde{\varphi}|^{p-1} |\widetilde{\varphi}(x) - \widetilde{\varphi}(y)|\,,
\end{split}
\end{equation}
where $\sigma$ is some value in $[0,1]$. Then using \eqref{eq:Caccioppoli:IEstimate2} and Young's Inequality,
\begin{equation}\label{eq:Caccioppoli:IEstimate3}
\begin{split}
\rmR_1 &= p \big( \sigma \widetilde{\varphi}(x) + (1-\sigma) \widetilde{\varphi}(y) \big)^{p-1} (\widetilde{\varphi}(x) - \widetilde{\varphi}(y)) |u(x)-u(y)|^{p-2} (u(x)-u(y)) u(y) \\
	&\geq - p | \widetilde{\varphi} |^{p-1} |\widetilde{\varphi}(x) - \widetilde{\varphi}(y)| |u(x)-u(y)|^{p-1} |u(y)|  \\
	&\geq - \frac{1}{p'} \varphi^q(x) |u(x)-u(y)|^p - p^{p-1} |\widetilde{\varphi}(x)-\widetilde{\varphi}(y)|^p |u(y)|^p\,.
\end{split}
\end{equation}
Combining \eqref{eq:Caccioppoli:IEstimate1} and \eqref{eq:Caccioppoli:IEstimate3} gives
\begin{equation}\label{eq:Caccioppoli:IEstimate4}
	\begin{split}
|u(x)-u(y)|^{p-2} (u(x)-u(y)) (\varphi^q(x)u(x) - \varphi^q(y) u(y)) &\geq C\varphi^q(x) |u(x) - u(y)|^p \\
&\qquad -  C'|\widetilde{\varphi}(x)-\widetilde{\varphi}(y)|^p |u(y)|^p
\end{split}
\end{equation}
in the case that $\varphi(x) \geq \varphi(y)$. Now we assume that $\varphi(y) \geq \varphi(x)$. By adding and subtracting $\varphi^q(y) u(x)$ and proceeding similarly to the first case,
{\small \begin{equation}\label{eq:Caccioppoli:IEstimate5}
|u(x)-u(y)|^{p-2} (u(x)-u(y)) (\varphi^q(x)u(x) - \varphi^q(y) u(y)) \geq C \varphi^q(y) |u(x) - u(y)|^p -  |\widetilde{\varphi}(x)-\widetilde{\varphi}(y)|^p |u(x)|^p\,.
\end{equation}}
Using symmetry and the estimates \eqref{eq:Caccioppoli:IEstimate4} and \eqref{eq:Caccioppoli:IEstimate5} gives
\begin{equation}\label{eq:Caccioppoli:IEstimate6}
\rmI_1 \geq C \iintdm{B}{B}{\frac{|u(x)-u(y)|^p}{|x-y|^{n+sp}} \max \{\varphi^q(x),\varphi^q(y) \} }{y}{x} - C' \iintdm{B}{B}{\frac{|\widetilde{\varphi}(x) - \widetilde{\varphi}(y)|^p}{|x-y|^{n+sp}} |u(x)|^p }{y}{x} \,,
\end{equation}
where $C$ and $C'$ depend only on $p$. Finally, since
\begin{equation*}
 \left| \widetilde{\varphi}(x) u(x) - \widetilde{\varphi}(y) u(y) \right|^p \leq 2^{p-1} \varphi^q(y) \left| u(x) - u(y)  \right|^p + 2^{p-1} |u(x)|^p |\widetilde{\varphi}(x) - \widetilde{\varphi}(y)|^p
\end{equation*}
we obtain
\begin{equation}\label{eq:Caccioppoli:IEstimate7}
\rmI_1 \geq C \iintdm{B}{B}{\frac{|\widetilde{\varphi}(x) u(x) - \widetilde{\varphi}(y) u(y)|^p}{|x-y|^{n+sp}} }{y}{x} - C \iintdm{B}{B}{\frac{|\widetilde{\varphi}(x) - \widetilde{\varphi}(y)|^p}{|x-y|^{n+sp}} |u(x)|^p }{y}{x} \,.
\end{equation}
We then proceed in exactly a similar way to bound $\rmI_2$, with $q$ taking the role of $p$ and $\varphi$ taking the role of $\widetilde{\varphi}$. The resulting estimate is
\begin{equation}\label{eq:Caccioppoli:IEstimate8}
\rmI_2 \geq C \iintdm{B}{B}{a(x,y) \frac{|\varphi(x) u(x) - \varphi(y) u(y)|^q}{|x-y|^{n+tq}} }{y}{x} - C \iintdm{B}{B}{a(x,y) \frac{|\varphi(x) - \varphi(y)|^q}{|x-y|^{n+tq}} |u(x)|^q }{y}{x}\,.
\end{equation}

\noindent{\bf Estimate of $\mathrm{II}$.} Directly, we have 
\begin{align*}
|u(x)-u(y)|^{p-2} (u(x)-u(y)) \varphi^q(x) u(x)  &\geq -|u(x)-u(y)|^{p-1}  \varphi^q(x) |u(x)|\\
&\geq -2^{p-2}\left(\varphi^q(x) |u(x)|^{p} +  \varphi^q(x) |u(x)||u(y)|^{p-1}\right)
\end{align*}
where the last inequality follows from the algebraic inequality $(a+b)^{p-1} \leq 2^{p-2}(a^{p-1} + b^{p-1})$ (valid since $p\geq 2$).  Similarly, we have 
\[
|u(x)-u(y)|^{q-2} (u(x)-u(y)) \varphi^q(x) u(x) \geq -2^{q-2}\left(\varphi^q(x) |u(x)|^{q} +  \varphi^q(x) |u(x)||u(y)|^{q-1}\right).
\]
We then estimate {\rm II} as 
\begin{equation}\label{eq:Caccioppoli:IIEstimate2}
\begin{split}
\mathrm{II}&\geq - C \iintdm{B}{\bbR^n \setminus B}{\left[\varphi^q(x) \frac{|u(x)|^{p} + |u(y)|^{p-1}|u(x)|}{|x-y|^{n+sp}}\right.\\
&\qquad+\left.a(x,y) \varphi^q(x) \frac{|u(x)|^{q}  + |u(y)|^{q-1}|u(x)|}{|x-y|^{n+tq}}\right]}{y}{x}\,.
\end{split}
\end{equation}
Now, for every $x \in \supp \varphi \subset \frac{1}{2}B$ and every $y \in \bbR^n \setminus B$
\begin{equation*}
\begin{split}
|x - y| &\geq |y - x_0| - |x-x_0| \geq R - \frac{R}{2} = \frac{R}{2}\,,
\end{split}
\end{equation*}
so
\begin{equation*}
\frac{|x_0-y|}{|x-y|} \leq \frac{|x_0 - x| + |x-y|}{|x-y|} = 1 + \frac{|x_0 - x|}{|x-y|} \leq 2\,.
\end{equation*}
Thus we can replace $|x-y|$ with $|x_0 - y|$ in \eqref{eq:Caccioppoli:IIEstimate2}, which gives
\begin{equation}\label{eq:Caccioppoli:IIEstimate3}
\begin{split}
\mathrm{II}&\geq - C \iintdm{B}{\bbR^n \setminus B}{\left[\varphi^q(x) \frac{|u(x)|^{p} + |u(y)|^{p-1}|u(x)|}{|x_0-y|^{n+sp}}\right.\\
&\qquad+\left.a(x,y) \varphi^q(x) \frac{|u(x)|^{q}  + |u(y)|^{q-1}|u(x)|}{|x_0-y|^{n+tq}}\right]}{y}{x}\,\\
&\geq -C\left[\iintdm{B}{\bbR^n \setminus B}{\varphi^q(x) \frac{|u(x)|^{p} + |u(y)|^{p-1}|u(x)| }{|x_0-y|^{n+sp}}}{y}{x}\right.\\
&\qquad+\left.\iintdm{B}{\bbR^n \setminus B}{\varphi^q(x)a(x,y)\frac{|u(x)|^{q} + |u(y)|^{q-1}|u(x)|}{|x_0-y|^{n+tq}}}{y}{x}\right]
\end{split}
\end{equation}
Simplifying further, we obtain 
\[
\begin{split}
\mathrm{II}&\geq - {C\over R^{sp}} \intdm{B}{\varphi^q(x) |u(x)|^{p}}{x} - C \iintdm{B}{\bbR^n \setminus B}{\varphi^q(x)|u(x)| \frac{|u(y)|^{p-1}}{|x_0-y|^{n+sp}}}{y}{x}\\
&\qquad-C \iintdm{B}{\bbR^n \setminus B}{a(x,y) \varphi^q(x) \frac{|u(x)|^{q}  + |u(y)|^{q-1}|u(x)|}{|x_0-y|^{n+tq}}}{y}{x}.\,
\end{split}
\]
Finally we estimate the right hand side $\int_{B} \varphi^{q} f(x) u(x)dx$. To that end, by H\"older's inequality and since $|\varphi| \leq 1$
\begin{equation*}
\begin{split}
\int_{B} \varphi^{q} f(x) u(x)dx&\leq \left( \intdm{B}{|\widetilde{\varphi}(x) u(x)|^{p^*}}{x} \right)^{1/p^*} \left( \intdm{B}{|f(x)|^{p_*}}{x} \right)^{1/p_*} \\
	& = R^n \left( \fint_B |\widetilde{\varphi}(x) u(x)|^{p^*} \, \rmd x \right)^{1/p^*} \left( \fint_B |f(x)|^{p_*} \, \rmd x \right)^{1/p_*}\,.
\end{split}
\end{equation*}
Apply the Poincar\'e-Sobolev inequality (Theorem  \ref{thm:SobolevInequality}) to $\widetilde{\varphi} u$ to obtain
\begin{equation*}
\begin{split}
\int_{B} \varphi^{q} f(x) u(x)dx&\leq C R^{n/p'+s} \left( \int_B \int_B \frac{|\widetilde{\varphi}(x) u(x) - \widetilde{\varphi}(y) u(y)|^p}{|x-y|^{n+sp}} \, \rmd y \, \rmd x \right)^{1/p} \left( \fint_B |f(x)|^{p_*} \, \rmd x \right)^{1/p_*}\,.
\end{split}
\end{equation*}
By Young's inequality with $\sigma \in (0,1)$ suitably small, 
\begin{equation}\label{eq:Caccioppoli:IVEstimate1}
\int_{B} \varphi^{q} f(x) u(x)dx\leq \frac{C}{\sigma} R^{n+sp'} \left( \fint_B |f(x)|^{p_*} \, \rmd x \right)^{p'/p_*} + \sigma \int_B \int_B \frac{|\widetilde{\varphi}(x) u(x) - \widetilde{\varphi}(y) u(y)|^p}{|x-y|^{n+sp}} \, \rmd y \, \rmd x\,.
\end{equation}
We put together \eqref{eq:Caccioppoli:IEstimate7}, \eqref{eq:Caccioppoli:IEstimate8}, \eqref{eq:Caccioppoli:IIEstimate3}, and \eqref{eq:Caccioppoli:IVEstimate1}, and using the symmetry of $a$, we conclude that there exists $C = C(\texttt{data})$ and an arbitrarily small $\sigma \in (0,1)$ such that
\begin{equation}\label{eq:Caccioppoli:MajorEstimate1}
\begin{split}
& \iintdm{B}{B}{ \frac{|\widetilde{\varphi}(x) u(x) - \widetilde{\varphi}(y) u(y)|^p}{|x-y|^{n+sp}} }{y}{x} 
+  \iintdm{B}{B}{a(x,y) \frac{|\varphi(x) u(x) - \varphi(y) u(y)|^q}{|x-y|^{n+tq}} }{y}{x} \\
	&\leq C \iintdm{B}{B}{\frac{|\widetilde{\varphi}(x) - \widetilde{\varphi}(y)|^p}{|x-y|^{n+sp}} |u(x)|^p }{y}{x} + C \iintdm{B}{B}{a(x,y) \frac{|\varphi(x) - \varphi(y)|^q}{|x-y|^{n+tq}} |u(x)|^q }{y}{x}\\
	& + {C\over R^{sp}} \intdm{B}{\varphi^q(x) |u(x)|^{p}}{x} + C \iintdm{B}{\bbR^n \setminus B}{\varphi^q(x)|u(x)| \frac{|u(y)|^{p-1}}{|x_0-y|^{n+sp}}}{y}{x}\\
&\qquad C \iintdm{B}{\bbR^n \setminus B}{a(x,y) \varphi^q(x) \frac{|u(x)|^{q}  + |u(y)|^{q-1}|u(x)|}{|x_0-y|^{n+tq}}}{y}{x}.\,\\
	&\quad + \frac{C}{\sigma} R^{n+sp'} \left( \fint_B |f(x)|^{p_*} \, \rmd x \right)^{p'/p_*} + \sigma \int_B \int_B \frac{|\widetilde{\varphi}(x) u(x) - \widetilde{\varphi}(y) u(y)|^p}{|x-y|^{n+sp}} \, \rmd y \, \rmd x\,. 
\end{split}
\end{equation}
Now, since $|\grad \widetilde{\varphi}| \leq |\grad \varphi| \leq \frac{C}{R}$ the first integral on the right-hand side of \eqref{eq:Caccioppoli:MajorEstimate1} can be majorized by
\begin{equation}\label{eq:Caccioppoli:Piece1}
\frac{C}{R^p} \int_B |u(x)|^p \int_{B}  |x-y|^{-n+(1-s)p} \, \rmd y \, \rmd x \leq \frac{C}{R^{sp}} \int_B |u(x)|^p \, \rmd x\,,
\end{equation}
where the constant $C$ is independent of $R$. We also use the fact that $\varphi \in (0, 1)$ and choosing $\sigma \in (0, 1)$ to absorb the last term on the right-hand side we obtain 
\begin{equation}\label{eq:Caccioppoli:MajorEstimate2}
\begin{split}
& \iintdm{B}{B}{ \frac{|\widetilde{\varphi}(x) u(x) - \widetilde{\varphi}(y) u(y)|^p}{|x-y|^{n+sp}} }{y}{x} 
+  \iintdm{B}{B}{a(x,y) \frac{|\varphi(x) u(x) - \varphi(y) u(y)|^q}{|x-y|^{n+tq}} }{y}{x} \\
	&\leq \frac{C}{R^{sp}} \int_B |u(x)|^p \, \rmd x + C \iintdm{B}{B}{a(x,y) \frac{|\varphi(x) - \varphi(y)|^q}{|x-y|^{n+tq}} |u(x)|^q }{y}{x}\\
	&+ C \intdm{B}{\varphi^q(x) |u(x)| }{x} \intdm{\bbR^n \setminus B}{\frac{|u(y)|^{p-1}}{|x_0-y|^{n+sp}}}{y} \\
	&+C \iintdm{B}{\bbR^n \setminus B}{a(x,y) \varphi^q(x) \frac{|u(x)|^{q} + |u(y)|^{q-1}|u(x)|}{|x_0-y|^{n+tq}}}{y}{x} + \frac{C}{\sigma} R^{n+sp'} \left( \fint_B |f(x)|^{p_*} \, \rmd x \right)^{p'/p_*}\,.
\end{split}
\end{equation}
That concludes the proof. 
\end{proof}

\begin{remark} As a follow up to Remark \ref{remarka=0}, in the event $a\equiv 0$, then the assumption $u\in L^{\infty}$ is not necessary for the validity  of the Caccioppoli inequality. 
\end{remark}

\section{Reverse H\"older Inequality}\label{sec:RevHolder}

\subsection{The Dual Pair}

We summarize some basic properties of the measure $\nu$ defined in \eqref{eq:MeasureDefn}. These properties are natural extensions of those established in \cite[Proposition 4.1]{kuusi2015}
\begin{theorem}\label{thm:Measure}
For any $\veps \in (0,1/p)$, the measure $\nu$ defined as
\begin{equation*}
\nu(\cA) := \intdm{\cA}{\frac{1}{|x-y|^{n-\veps p}}}{y}\, \rmd x\,, \qquad \cA \subset \bbR^{2n}\,,
\end{equation*}
is absolutely continuous with respect to Lebesgue measure on $\bbR^{2n}$. Additionally,
\begin{itemize}
\item For $\cB = B_R(x_0) \times B_R(x_0)$,
\begin{equation}\label{eq:MeasureOfBall}
\nu(\cB) = \frac{c(n,p,\veps) R^{n+\veps p}}{\veps}\,,
\end{equation}
where $c(n,p,\veps)$ is a constant depending only on $n$, $p$ and $\veps$ that satisfies $1/\widetilde{c}(n,p) \leq c(n,p,\veps) \leq \widetilde{c}(n,p)$, where $\widetilde{c}$ is another constant depending only on $n$ and $p$. 

\item For every $x \in \bbR^n$ and for $R \geq r > 0$, \begin{equation}\label{eq:MeasureDoublingProperty}
\frac{\nu( \cB(x,R))}{\nu( \cB(x,r))} = \left( \frac{R}{r} \right)^{n+\veps p}\,.
\end{equation}

\item For every $a \leq 1$, $R> 0$ and $x \in \bbR^n$, there exists a constant $C_d = C_d(n,p)$ such that
\begin{equation}\label{eq:MeasureOfBallAndCube}
\frac{\nu(\cB(x,R))}{\nu(K_1 \times K_2)} \leq \frac{C_d}{a^{2n} \veps}
\end{equation}
for any two cubes $K_1$, $K_2 \subset B_R(x)$ with sides parallel to the coordinate axes and such that $|K_1| = |K_2| = (aR)^n$.
\end{itemize}
\end{theorem}
\begin{proof}
The identity \eqref{eq:MeasureOfBall} follows from the definition of the measure $\nu$ and a scaling argument. Indeed, for $r=1$ and $x = 0$
\begin{equation*}
\begin{split}
\nu(\cB(0,1)) &= \iintdm{B_1(0)}{B_1(0)}{\frac{1}{|x-y|^{n-\veps p}}}{y}{x} = \iintdm{B_1(0)}{B_1(x)}{\frac{1}{|y|^{n-\veps p}}}{y}{x} \\
	&\leq \iintdm{B_1(0)}{B_2(0)}{\frac{1}{|y|^{n-\veps p}}}{y}{x} = \frac{(\omega_{n-1})^2 2^{\veps p}}{\veps p}\,,
\end{split}
\end{equation*}
and on the other hand since $B_{\sfrac{1}{2}}(0) \subset B_1(x)$ for every $x \in B_{\sfrac{1}{2}}(0)$
\begin{equation*}
\begin{split}
\nu(\cB(0,1)) = \iintdm{B_1(0)}{B_1(x)}{\frac{1}{|y|^{n-\veps p}}}{y}{x} &\geq \iintdm{B_{1/2}(0)}{B_1(x)}{\frac{1}{|y|^{n-\veps p}}}{y}{x} \\
	&\geq \iintdm{B_{1/2}(0)}{B_{1/2}(0)}{\frac{1}{|y|^{n-\veps p}}}{y}{x} = \frac{\omega_{n-1}}{n 2^n}\frac{\omega_{n-1}}{ 2^{\veps p} \veps p}\,.
\end{split}
\end{equation*}
Thus $\nu(\cB(0,1)) = \veps^{-1} c(n,p,\veps)$, where $ \frac{(\omega_{n-1})^2}{np 2^{n+p}}\leq  c(n,p,\veps) \leq  \frac{(\omega_{n-1})^2 2^p}{p}$. Then a scaling and translation argument gives \eqref{eq:MeasureOfBall}. The doubling property \eqref{eq:MeasureDoublingProperty} follows from \eqref{eq:MeasureOfBall}. To see \eqref{eq:MeasureOfBallAndCube}, note that $|x-y| < 2R$ for $x \in K_1$ and $y \in K_2$ since $K_1$ and $K_2 \subset B_R(x)$. Thus by \eqref{eq:MeasureOfBall}
\begin{equation*}
\begin{split}
\nu(\cB(x,R)) = \frac{c(n,p,\veps)}{\veps} R^{n+\veps p} &= \frac{c(n,p,\veps)}{\veps} \frac{1}{a^{2n} R^{n-\veps p}} \iintdm{K_1}{K_2}{}{x}{y} \\
&\leq \frac{C(n,p)}{a^{2n} \veps} \iintdm{K_1}{K_2}{\frac{1}{|x-y|^{n-\veps p}}}{y}{x} = \frac{C_d}{a^{2n}\veps} \nu(K_1 \times K_2)\,,
\end{split}
\end{equation*}
which is \eqref{eq:MeasureOfBallAndCube}.
\end{proof}

\subsection{Reverse H\"older Inequality} Recall that 
\begin{equation}\label{eq:DefnOfUandF}
U(x,y) = \frac{|u(x)-u(y)|}{|x-y|^{s+\veps}}\,, \qquad\text{and define\,\,} F(x,y) := |f(x)|\,.
\end{equation}
Then $F \in L^{p_* + \delta}_{loc}(\bbR^{2n})$ for every $\delta \in (0,\delta_0)$, as a direct calculation using Theorem \ref{thm:Measure} shows.

We now report the compatibility of the Sobolev-Poincar\'e inequality with the definition of $U$. Given $B = B_R(x_0)$, define $\tau \in (0,1)$, and $\eta \in (1,\infty)$ to be differentiability and integrability constants respectively that have yet to be fixed. Letting $ \textstyle \veps \in (0,\min\{{s\over p}, 1-s\})$ and using \eqref{eq:MeasureOfBall},
\begin{equation*}
\fint_B \int_B \frac{\left| u(x) - u(y) \right|^{\eta}}{|x-y|^{n+\tau {\eta}}} \, \rmd y \, \rmd x =  \frac{C R^{\veps p}}{\veps}\fint_{\cB} U^{\eta} \, \rmd \nu
\end{equation*}
so long as
\begin{equation*}
\tau + \frac{\veps p}{\eta} = s + \veps\,.
\end{equation*}
Since $ \textstyle \veps \in (0,{s\over p})$ and $\veps < 1-s$ the exponent $\tau$ remains in $(0,1)$ for every $\eta \in (1,\infty)$. With this choice of $\tau$, by the fractional Sobolev inequality
\begin{equation*}
\left( \fint_B \left| \frac{u(x) - (u)_B}{R^{\tau}} \right|^m \, \rmd x \right)^{1/m} \leq C \left( \fint_B \int_B \frac{\left| u(x) - u(y) \right|^{\eta}}{|x-y|^{n+\tau \eta}} \, \rmd y \, \rmd x \right)^{1/\eta}
\end{equation*}
for every $m \in [1,\eta^{*_{\tau}}]$ with $\eta \in (1,\infty)$.
We choose $\eta$ to satisfy the relation
\begin{equation}\label{eq:DefnOfEta}
p = \eta^{*_{\tau}} = \frac{n \eta}{n - \tau \eta} = \frac{n \eta}{n- \eta(s + \veps - \frac{\veps p}{\eta})} \quad \Longleftrightarrow \quad \eta = \frac{np  +  \veps p^2}{n + sp + \veps p}\,.
\end{equation}
This choice of $\eta$ is a valid Lebesgue exponent; note that $\eta < p$ for all $n \geq 2$ and for all $p \in (1,\infty)$, and that $\eta > 1$ so long as $p \geq 2$.
Taking $m=\eta^{*_{\tau}}$ we summarize this discussion in the following lemma:
\begin{lemma}\label{lma:SobolevEmbeddingForDualPair}
Let $\veps \in (0,s/p)$ with $\veps < 1-s$ and $p \geq 2$.
Define $\eta = \frac{np  +  \veps p^2}{n + sp + \veps p}$. Then 
\begin{equation*}
\left( \fint_B \left| u(x) - (u)_B \right|^p \, \rmd x \right)^{1/p} \leq \frac{C R^{s+\veps}}{\veps^{1/\eta}} \left( \fint_{\cB} U^{\eta} \, \rmd \nu \right)^{1/\eta}\,,
\end{equation*}
where $C = C(n,s,p)$. The same inequality holds when the ball $B$ is replaced by a cube $Q$ with sides of length $R$ and with $\cB$ replaced by $Q \times Q$.
\end{lemma}


Recall that $G(x, y, U) =U^p + A(x,y) U^q$. We have the following $L^1_{loc}$ estimate for $G$ which will lead us to a scale-invariant reverse H\"older's inequality.  
\begin{proposition}\label{prop:RevHolder}
Let $p \in [2,\infty)$, and let $\veps < 1-s$ with $ \textstyle \veps \in ( 0, \min \{ s( \frac{tq}{sp} - \frac{1}{p'}) , \frac{s}{p} \} )$. (This choice is possible by Assumption \ref{Assumption:Intro:Exp}). Let $\eta$ be given by the formula  in \eqref{eq:DefnOfEta}, Let $B = B_R(x_0)$ be a ball with $R \leq 1$. Then there exists a constant $C$ depending only on $\texttt{data}$ such that for any solution $u \in W^{s,p}(\bbR^n) \cap L^{\infty}(\bbR^n)$ to \eqref{eq:Intro:MainEqn} and for any $\sigma \in (0,1)$
\begin{equation}\label{eq:RevHolder}
\begin{split}
\left( \fint_{\frac{1}{4}\cB} G(x,y,U) \, \rmd \nu \right)^{1/p} &\leq \frac{C}{\veps^{1/\eta - 1/p}}\Bigg[\frac{1}{\sigma } \left( \fint_{\cB} U^{\eta} \, \rmd \nu \right)^{1/\eta} \\
	&\quad + { \sigma}\sum_{k=0}^{\infty} 
	\big( 2^{-k (\frac{sp}{p-1} -s - \veps)} +  2^{-k (\frac{tq}{p-1} -s - \veps)} \big) 
	\left( \fint_{2^k \cB} U^{\eta} \, \rmd \nu \right)^{1/\eta}\Bigg]\\
	&+ \frac{C [\veps \nu(\cB)]^{\frac{\theta}{p-1}}}{\veps^{  (1/p_* - 1/p') \frac{1}{p-1}  }} \left[ \left( \fint_{\cB} F^{p_*} \, \rmd \nu \right)^{1/p_*} \right]^{1/(p-1)}\,, \\
\end{split}
\end{equation}
where
%
\begin{equation*}
\theta := \frac{s-\veps(p-1)}{n+\veps p} > 0\,.
\end{equation*}
\end{proposition}
\begin{proof}
Let  $ \varphi \in C^{\infty}_c( B)$ such that $0 \leq \varphi \leq 1$, $\supp \varphi \subset \frac{1}{2} B$ and $|\grad \varphi| \leq \frac{C(n)}{R}$.
The function $(u-(u)_B) \varphi^q$ is also an admissible test function, and repeating the argument of Theorem \ref{thm:CaccioppoliEstimate} with this function instead of $\varphi^q u$ leads to the inequality \eqref{eq:CaccioppoliEstimate} with $u$ replaced with $u-(u)_B$.
Choosing the cutoff function to additionally satisfy $\varphi \equiv 1$ on $\frac{1}{4} B$, we can estimate
after dividing both sides of \eqref{eq:CaccioppoliEstimate} by $|B|$ as 
\[
\begin{split}
\frac{R^{\veps p}}{\veps} \int_{\frac{1}{4} \cB}G(x, y, U) \rmd \nu&=C  \frac{1}{|B|} \int_{\frac{1}{4} \cB} U^p \, \rmd \nu+ C  \frac{1}{|B|} \int_{\frac{1}{4} \cB} A(x,y) U^q\\
 &\leq C\fint_{B} \int_{B}  \frac{|\varphi^{q/p}(x) (u(x)-(u)_B) - \varphi^{q/p}(y) (u(y)-(u)_B)|^p}{|x-y|^{n+sp}}  \, \rmd y \, \rmd x \\
	& \quad+ C\fint_B \int_B {a(x,y) \frac{|\varphi(x) (u(x)-(u)_B) - \varphi(y) (u(y)-(u)_B)|^q}{|x-y|^{n+tq}} } \, \rmd y \, \rmd x. 
	\end{split}
	\]
	In then follows that 
	\begin{equation}\label{eq:RevHolderPiece0}
	\begin{split}
	\frac{R^{\veps p}}{\veps} \int_{\frac{1}{4} \cB}G(x, y, U) \rmd \nu&\leq  \rmI_1 + \rmI_2 + \rmI_3 + \rmI_4 + \rmI_5
	\end{split}
\end{equation}
	where 
	\[
	\begin{split}
	\rmI_1 =&\frac{C}{R^{sp}} \fint_B |u(x)-(u)_B|^p \, \rmd x \\
	\rmI_2=& C \fint_{B}{\varphi^q(x) |u(x)-(u)_B| } \, \rmd x \intdm{\bbR^n \setminus B}{\frac{|u(y)-(u)_B|^{p-1}}{|x_0-y|^{n+sp}}}{y}  \\
	\rmI_3  =& C \fint_{B} \int_{B}{a(x,y) \frac{|\varphi(x) - \varphi(y)|^q}{|x-y|^{n+tq}} |u(x)-(u)_B|^q } \, \rmd y \, \rmd x \\
	\rmI_4 = &C \fint_{B} \int_{\bbR^n \setminus B}{a(x,y) \varphi^q(x) \frac{|u(x)-(u)_B|^{q} + |u(y)-(u)_B|^{q-1}|u(x)-(u)_B|}{|x_0-y|^{n+tq}}} \, \rmd {y} \, \rmd {x} \\
	\rmI_5=&C R^{sp'} \left( \fint_B |f(x)|^{p_*} \, \rmd x \right)^{p'/p_*} \,.
\end{split}
\]
In what follows, we estimate $\rmI_{i}$ for $i \in \{1, 2, 3,4, 5\}$. 

\noindent\textit{Estimate of $\rmI_1$:} Using Lemma \ref{lma:SobolevEmbeddingForDualPair}
\begin{equation}\label{eq:RevHolderPiece4}
\rmI_1 \leq \frac{C R^{\veps p} }{\veps^{p/\eta}}  \left( \fint_{\cB} U^{\eta} \, \rmd \nu \right)^{p/\eta}\,.
\end{equation}
\textit{Estimate of $\rmI_2$:} we write $\rmI_2$ as a product $\rmI_2 = \mathrm{II}_1 \cdot \rm{II}_2$, split the second integral $\rm{II}_2$ into annuli, and obtain
\begin{equation}\label{eq:RevHolderTelescope1}
\begin{split}
\mathrm{II}_2 &\leq \sum_{j = 0}^{\infty} \intdm{2^{j+1}B \setminus 2^j B}{\frac{|u(y) -(u)_B|^{p-1}}{|x_0 - y|^{n+sp}}}{y} \\
	&\leq \sum_{j = 0}^{\infty}  (2^j R)^{-n-sp} \intdm{2^{j+1}B \setminus 2^j B}{|u(y) -(u)_B|^{p-1}}{y} \\
	&\leq C(n) \sum_{j = 0}^{\infty}  (2^j R)^{-sp} \left[ \left( \fint_{2^{j+1}B} |u(y) -(u)_B|^{p-1} \, \rmd y \right)^{1/(p-1)} \right]^{p-1} \,.
\end{split}
\end{equation}
We write each integrand in the last line of \eqref{eq:RevHolderTelescope1} as a telescoping sum and use the triangle inequality in $L^{p-1}$ (note that $p \geq 2$ here) to obtain
\begin{equation*}
\begin{split}
 \left( \fint_{2^{j+1}B} |u(y) -(u)_B|^{p-1} \, \rmd y \right)^{1/(p-1)} &=  \bigg( \fint_{2^{j+1}B} |u(y) - (u)_{2^{j+1}B} + (u)_{2^{j+1}B} \\
	&\quad - (u)_{2^{j}B} + (u)_{2^{j}B} - \ldots + (u)_{2B} - (u)_B|^{p-1} \, \rmd y \bigg)^{1/(p-1)}\\
	&\leq \left( \fint_{2^{j+1}B} |u - (u)_{2^{j+1}B}|^{p-1} \, \rmd y \right)^{1/(p-1)} + \sum_{k=0}^{j} |(u)_{2^{k+1}B} - (u)_{2^{k}B}|\,.
\end{split}
\end{equation*}
By H\"older's inequality, for each $k \in \{ 0,1, \ldots, j \}$
\begin{equation*}
|(u)_{2^{k+1}B} - (u)_{2^{k}B}| \leq \fint_{2^k B} | u(y) - (u)_{2^{k+1}B}| \, \rmd y \leq \left( \fint_{2^{k+1} B} | u(y) - (u)_{2^{k+1}B}|^{p} \, \rmd y \right)^{1/p}\,.
\end{equation*}
Therefore, additionally using H\"older's inequality on $ \left( \fint_{2^{j+1}B} |u - (u)_{2^{j+1}B}|^{p-1} \, \rmd y \right)^{1/(p-1)}$,
\begin{equation}\label{eq:RevHolderTelescope2}
\left( \fint_{2^{j+1}B} |u(y) -(u)_B|^{p-1} \, \rmd y \right)^{1/(p-1)}  \leq 2 \sum_{k=0}^{j+1} \left( \fint_{2^k B} | u(y) - (u)_{2^k B}|^{p} \, \rmd y \right)^{1/p}\,.
\end{equation}
Apply the Sobolev Embedding lemma (Lemma \ref{lma:SobolevEmbeddingForDualPair}) to each term to get
\begin{equation*}
\left( \fint_{2^k B} | u(y) - (u)_{2^k B}|^{p} \, \rmd y \right)^{1/p} \leq \frac{C(2^k R)^{s+\veps}}{\veps^{1/\eta}} \left( \fint_{2^k \cB} U^{\eta} \, \rmd \nu \right)^{1/\eta}\,.
\end{equation*}
Combining the last display with \eqref{eq:RevHolderTelescope2} gives
\begin{equation}\label{eq:RevHolder:Telescope:p-1}
\left( \fint_{2^{j+1}B} |u(y) -(u)_B|^{p-1} \, \rmd y \right)^{1/(p-1)}  \leq \frac{C R^{s+\veps}}{\veps^{1/\eta}} \sum_{k=0}^{j+1} 2^{k(s + \veps)} \left( \fint_{2^k \cB} U^{\eta} \, \rmd \nu \right)^{1/\eta}\,,
\end{equation}
and combining the previous line with \eqref{eq:RevHolderTelescope1} gives
\begin{equation*}
\mathrm{II}_2 \leq \frac{C R^{-s + \veps(p-1)}}{\veps^{(p-1)/\eta}} \sum_{j = 0}^{\infty}  2^{-jsp}  \left[ \sum_{k=0}^{j+1} 2^{k(s + \veps)} \left( \fint_{2^k \cB} U^{\eta} \, \rmd \nu \right)^{1/\eta} \right]^{p-1}\,.
\end{equation*}
Using Minkowski's inequality on the sums, 
\begin{equation*}
\begin{split}
& \left\{ \sum_{j = 0}^{\infty}   \left[ \sum_{k=0}^{j+1} 2^{\frac{-jsp}{p-1}} 2^{k(s + \veps)} \left( \fint_{2^k \cB} U^{\eta} \, \rmd \nu \right)^{1/\eta} \right]^{p-1} \right\}^{(p-1)/(p-1)} \\
	 &\leq \Bigg\{ \left( \fint_{\cB} U^{\eta} \, \rmd \nu \right)^{1/\eta} \left( \sum_{j=0}^{\infty} 2^{-jsp} \right)^{1/(p-1)} \\
	 &\qquad + \sum_{k = 1}^{\infty} \left( \sum_{j = k-1}^{\infty} 2^{-jsp + k(p-1)(s + \veps)} \left( \fint_{2^k \cB} U^{\eta} \, \rmd \nu \right)^{(p-1)/\eta} \right)^{1/(p-1)}  \Bigg\}^{p-1} \\
	 &\leq C \Bigg\{ \left( \fint_{\cB} U^{\eta} \, \rmd \nu \right)^{1/\eta} + \sum_{k = 1}^{\infty} \left( \fint_{2^k \cB} U^{\eta} \, \rmd \nu \right)^{1/\eta} \left( \sum_{j = k-1}^{\infty} 2^{-jsp + k(p-1)(s + \veps)}  \right)^{1/(p-1)}  \Bigg\}^{p-1}\,,
\end{split}
\end{equation*}
where $C = C(s,p)$. Using  \eqref{eq:GeometricSeriesEstimate}, we estimate the second term of the right hand side as 
\begin{equation*}
\begin{split}
&\sum_{k = 1}^{\infty} \left( \fint_{2^k \cB} U^{\eta} \, \rmd \nu \right)^{1/\eta} \left( \sum_{j = k-1}^{\infty} 2^{-jsp + k(p-1)(s + \veps)}  \right)^{1/(p-1)} \\
	&\leq C \sum_{k = 1}^{\infty} \left( \fint_{2^k \cB} U^{\eta} \, \rmd \nu \right)^{1/\eta} \left( 2^{-ksp + k(p-1)(s + \veps)}  \right)^{1/(p-1)} 
	= C \sum_{k = 1}^{\infty} 2^{-k(\frac{s}{p-1} - \veps)} \left( \fint_{2^k \cB} U^{\eta} \, \rmd \nu \right)^{1/\eta}\,,
\end{split}
\end{equation*}
where $C = C(s,p)$. 
Combining the previous three displays gives
\begin{equation*}
\mathrm{II}_2 \leq \frac{CR^{-s+\veps(p-1)}}{\veps^{(p-1)/\eta}} \left( \sum_{k = 0}^{\infty} 2^{-k(\frac{s}{p-1} - \veps)} \left( \fint_{2^k \cB} U^{\eta} \, \rmd \nu \right)^{1/\eta} \right)^{p-1}\,.
\end{equation*}
The first integral $\mathrm{II}_1$ can be estimated using H\"older's inequality and Lemma \ref{lma:SobolevEmbeddingForDualPair}:
\begin{equation*}
\mathrm{II}_1 \leq \left( \fint_{B} |u(x)-(u)_B|^p \, \rmd x \right)^{1/p} \leq \frac{C R^{s+\veps}}{\veps^{1/\eta}} \left( \fint_{\cB} U^{\eta} \, \rmd \nu \right)^{1/\eta}\,.
\end{equation*}
Combining the previous two displays and using that $\rmI_2 = \mathrm{II}_1 \cdot \mathrm{II}_2$,
\begin{equation*}
\begin{split}
\mathrm{I}_2 
		& \leq \frac{CR^{\veps p}}{\veps^{p/\eta}} \left( \fint_{\cB} U^{\eta} \, \rmd \nu \right)^{1/\eta} \left( \sum_{k = 0}^{\infty} 2^{-k(\frac{s}{p-1} - \veps)} \left( \fint_{2^k \cB} U^{\eta} \, \rmd \nu \right)^{1/\eta} \right)^{p-1}\,.
\end{split}
\end{equation*}
We conclude the estimate for $\mathrm{I}_2$ by applying Young's inequality for arbitrary $\sigma \in (0,1)$:
\begin{equation}\label{eq:RevHolderPiece5}
\mathrm{I}_2 \leq \frac{CR^{\veps p}}{\sigma^p \veps^{p/\eta}} \left( \fint_{\cB} U^{\eta} \, \rmd \nu \right)^{p/\eta} + \frac{C \sigma^p R^{\veps p}}{\veps^{p/\eta}} \left( \sum_{k = 0}^{\infty} 2^{-k(\frac{s}{p-1} - \veps)} \left( \fint_{2^k \cB} U^{\eta} \, \rmd \nu \right)^{1/\eta} \right)^p\,.
\end{equation}
\textit{Estimates of $\rmI_3$ and $\rmI_4$:}
With the estimates 
\begin{equation*}
\intdm{B(x_0,R)}{\frac{|\varphi(x)-\varphi(y)|^q}{|x-y|^{n+tq}}}{y} \leq  \intdm{B(x_0,R)}{\frac{\Vnorm{\grad \varphi}_{L^{\infty}(B)}}{|x-y|^{n+(t-1)q}}}{y} = \frac{C}{R^{tq}} \qquad \text{ for all } x \in B(x_0,R)
\end{equation*}
and
\begin{equation*}
\intdm{\bbR^n \setminus B(x_0,R)}{\frac{1}{|x_0-y|^{n+tq}}}{y} \leq \frac{C(n,t,q)}{R^{tq}}\,,
\end{equation*}
we get
\begin{equation*}
\begin{split}
\rmI_3 + \rmI_4 &\leq \frac{C M}{R^{tq}} \fint_B |u(x)-(u)_B|^q \, \rmd x + C M \fint_{B} {|u(x)-(u)_B| } \,\rmd x \intdm{\bbR^n \setminus B}{\frac{|u(y)-(u)_B|^{q-1}}{|x_0-y|^{n+tq}}}{y} \\
	&\leq C M R^{sp-tq} \Vnorm{u}_{L^{\infty}}^{q-p} \fint_B \frac{|u(x)-(u)_B|^p}{R^{sp}} \, \rmd x + C \Vnorm{u}_{L^{\infty}}^{q-p}  \fint_{B} {|u(x)-(u)_B| } \,\rmd x \intdm{\bbR^n \setminus B}{\frac{|u(y)-(u)_B|^{p-1}}{|x_0-y|^{n+tq}}}{y}  \\
	&\leq C \fint_B \frac{|u(x)-(u)_B|^p}{R^{sp}} \, \rmd x + C \fint_{B} {|u(x)-(u)_B| } \,\rmd x \intdm{\bbR^n \setminus B}{\frac{|u(y)-(u)_B|^{p-1}}{|x_0-y|^{n+tq}}}{y} := \wt{\rmI}_3 + \wt{\rmI}_4\,,
\end{split}
\end{equation*}
where $C=C(\texttt{data})$ in the last line. We additionally used that $sp \geq tq$ and $R \leq 1$.

We estimate $\wt{\rmI}_3$ indentically to $\rmI_1$ using \eqref{eq:RevHolderPiece4}:
\begin{equation}\label{eq:RevHolderPiece8}
\wt{\rmI}_3 \leq \frac{C R^{\veps p} }{\veps^{p/\eta}}  \left( \fint_{\cB} U^{\eta} \, \rmd \nu \right)^{p/\eta}\,.
\end{equation}
%
%
%
%
The estimate for $\wt{\rmI}_4$ is very similar to the estimate for $\rmI_2$. Write the product of the two integrals as $\wt{\rmI}_4 = \mathrm{IV}_1 \cdot \rm{IV}_2$, split the second integral $\rm{IV}_2$ into annuli, and obtain the analogue of \eqref{eq:RevHolderTelescope1}
\begin{equation*}
\mathrm{IV}_2 \leq C \sum_{j=0}^{\infty} (2^j R)^{-tq} \left[ \left( \fint_{2^{j+1} B} |u(y)-(u)_B|^{p-1} \, \rmd y \right)^{1/(p-1)} \right]^{p-1}\,.
\end{equation*}
Use the estimate \eqref{eq:RevHolder:Telescope:p-1} to get
\begin{equation*}
\mathrm{IV}_2 \leq \frac{C R^{-tq +s(p-1) +\veps (p-1)}}{\veps^{(p-1)/\eta}} \sum_{j=0}^{\infty} 2^{-jtq} \left[ \sum_{k=0}^{j+1} 2^{k(s + \veps)} \left( \fint_{2^k \cB} U^{\eta} \, \rmd \nu \right)^{1/\eta} \right]^{p-1}\,.
\end{equation*}
Apply Minkowski's inequality on the sums and use \eqref{eq:GeometricSeriesEstimate} in a way exactly similar to the estimate for $\mathrm{II}$:
\begin{equation*}
\begin{split}
& \sum_{j = 0}^{\infty}   \left[ \sum_{k=0}^{j+1} 2^{\frac{-jtq}{p-1}} 2^{k(s + \veps)} \left( \fint_{2^k \cB} U^{\eta} \, \rmd \nu \right)^{1/\eta} \right]^{p-1}  \\
	 &\leq C \Bigg\{ \left( \fint_{\cB} U^{\eta} \, \rmd \nu \right)^{1/\eta} + \sum_{k = 1}^{\infty} \left( \fint_{2^k \cB} U^{\eta} \, \rmd \nu \right)^{1/\eta} \left( \sum_{j = k-1}^{\infty} 2^{-jtq + k(p-1)(s + \veps)}  \right)^{1/(p-1)}  \Bigg\}^{p-1} \\
	 &\leq C \left\{ \sum_{k = 0}^{\infty} 2^{-k(\frac{tq}{p-1} - s - \veps)} \left( \fint_{2^k \cB} U^{\eta} \, \rmd \nu \right)^{1/\eta} \right\}^{p-1}\,.
\end{split}
\end{equation*}
Combining the previous two displays gives
\begin{equation*}
\mathrm{IV}_2 \leq \frac{CR^{-tq +s(p-1)+\veps(p-1)}}{\veps^{(p-1)/\eta}} \left( \sum_{k = 0}^{\infty}  2^{-k(\frac{tq}{p-1} -s -\veps )} \left( \fint_{2^k \cB} U^{\eta} \, \rmd \nu \right)^{1/\eta} \right)^{p-1}\,.
\end{equation*}
Now, since $sp \geq tq$ and since $R \leq 1$, it follows that $R^{-tq + s(p-1) + \veps(p-1)} \leq R^{-s + \veps(p-1)}$. 
Then by using H\"older's inequality and Lemma \ref{lma:SobolevEmbeddingForDualPair} to estimate $\mathrm{IV}_1$,
\begin{equation*}
\begin{split}
\wt{\rmI}_4 &\leq \frac{CR^{-tq +s(p-1)+\veps(p-1)}}{\veps^{(p-1)/\eta}} \left( \fint_B |u(x)-(u)_B| \, \rmd x \right) \left( \sum_{k = 0}^{\infty}  2^{-k(\frac{tq}{p-1} -s -\veps )} \left( \fint_{2^k \cB} U^{\eta} \, \rmd \nu \right)^{1/\eta} \right)^{p-1} \\
	& \leq \frac{CR^{\veps(p-1)}}{\veps^{(p-1)/\eta}} \left( \fint_B \left| \frac{u(x)-(u)_B}{R^s} \right|^p \, \rmd x \right)^{1/p} \left( \sum_{k = 0}^{\infty}  2^{-k(\frac{tq}{p-1} -s -\veps )} \left( \fint_{2^k \cB} U^{\eta} \, \rmd \nu \right)^{1/\eta} \right)^{p-1} \\
	& \leq \frac{CR^{\veps p}}{\veps^{p/\eta}} \left( \fint_{\cB} U^{\eta} \rmd x \right)^{1/\eta} \left( \sum_{k = 0}^{\infty}  2^{-k(\frac{tq}{p-1} -s -\veps )} \left( \fint_{2^k \cB} U^{\eta} \, \rmd \nu \right)^{1/\eta} \right)^{p-1}\,.
\end{split}
\end{equation*}
We conclude the estimate for $\wt{\rmI}_4$ by applying  Young's inequality for arbitrary $\sigma \in (0,1)$:
\begin{equation}\label{eq:RevHolderPiece7}
\wt{\rmI}_4 \leq \frac{CR^{\veps p}}{\sigma^p \veps^{p/\eta}} \left( \fint_{\cB} U^{\eta} \, \rmd \nu \right)^{p/\eta} + \frac{C \sigma^p R^{\veps p}}{\veps^{p/\eta}} \left( \sum_{k = 0}^{\infty} 2^{-k(\frac{tq}{p-1} - s - \veps)} \left( \fint_{2^k \cB} U^{\eta} \, \rmd \nu \right)^{1/\eta} \right)^p\,.
\end{equation}
\textit{Estimate of $\rmI_5$:} We use the definition of $\nu$ to get
\begin{equation*}
\begin{split}
\fint_B |f(x)|^{p_*} \, \rmd x &= \fint_B \fint_B |f(x)|^{p_*} \, \rmd y \, \rmd x \\
	&= \frac{1}{R^{n+\veps p}} \frac{1}{R^{n-\veps p}}\int_B \int_B |f(x)|^{p_*} \, \rmd y \, \rmd x  \\
	&\leq \frac{C(n)}{R^{n + \veps p}}\int_B \int_B \frac{|f(x)|^{p_*}}{|x-y|^{n-\veps p}} \, \rmd y \, \rmd x 
	\leq \frac{C}{\veps} \fint_{\cB} F^{p_*} \, \rmd \nu\,.
\end{split}
\end{equation*}
Therefore,
\begin{equation}\label{eq:RevHolderPiece9}
\mathrm{I}_5 \leq \frac{CR^{sp'}}{\veps^{p'/p_*}} \left( \fint_{\cB} F^{p_*} \, \rmd \nu \right)^{p' / p_*}\,.
\end{equation}
Combining  \eqref{eq:RevHolderPiece4}, \eqref{eq:RevHolderPiece5},  \eqref{eq:RevHolderPiece8}, \eqref{eq:RevHolderPiece7}, and  \eqref{eq:RevHolderPiece9} gives \eqref{eq:RevHolder} after some algebraic manipulations.
\end{proof}
\begin{remark}\label{remark-finiteness-sums} 
We make some remarks. The upper bound in \eqref{eq:RevHolder} can be simplified down to just one series. Since $sp \geq tq$
\begin{equation}\label{eq:SeriesComparisons}
\begin{split}
2^{-k(\frac{sp}{p-1} -s -\veps)} \leq 2^{-k(\frac{tq}{p-1} - s - \veps)}\,, \qquad k \in \bbZ_+\,,
\end{split}
\end{equation} 
so we can replace the infinite series on the right-hand side of \eqref{eq:RevHolder} with
\begin{equation*}
\frac{C \sigma}{\veps^{1/\eta-1/p}} \sum_{k = 0}^{\infty} \alpha_k \Vparen{\fint_{2^k \cB} U^{\eta} \, \rmd \nu }^{1/\eta}\,,
\end{equation*}
where
\begin{equation}\label{eq:DefnOfGeometricSeries}
\alpha_k := 2^{-k(\frac{tq}{p-1} - s - \veps)}\,.
\end{equation}
Moreover, again following up Remark \ref{remarka=0} in the case $a \equiv 0$ one simply takes $\alpha_k = 2^{-k(\frac{s}{p-1}-\veps)}$.
In any case, since $ \textstyle \veps \leq \min \{ s( \frac{tq}{sp} - \frac{1}{p'}) , \frac{s}{p} \} $ the series $\sum_{k=0}^{\infty}\alpha_k <\infty$ and as a consequence  
\[
\begin{split}
\sum_{k = 0}^{\infty} \alpha_k \Vparen{\fint_{2^k \cB} U^{\eta} \, \rmd \nu }^{1/\eta} &\leq \sum_{k = 0}^{\infty} \alpha_k \Vparen{\fint_{2^k \cB} U^{p} \, \rmd \nu }^{1/p}\\
&=C(\veps,p,s)\sum_{k = 0}^{\infty} \alpha_k  \Vparen{\int_{2^k B}\int_{2^{k}B} {|u(y)-u(x)|^{p}\over |x-y|^{n+sp}} \, \rmd x \rmd y }^{1/p}\\
& \leq  R^{-n/p-\epsilon}C(\veps,p,s)\Vparen{\int_{\mathbb{R}^{n}}\int_{\mathbb{R}^{n}} {|u(y)-u(x)|^{p}\over |x-y|^{n+sp}} \, \rmd x \rmd y }^{1/p} < \infty. 
\end{split}
\]
	
\end{remark}

The following corollary establishes a genuine scale-invariant reverse H\"older inequality for an appropriately scaled version of the integrand $G$. This quantity will satisfy a self-improving result. 
\begin{corollary}\label{cor:RevHolderFinal}
Let $\textstyle \veps \in \left( 0, \min \{ s( \frac{tq}{sp} - \frac{1}{p'}) , \frac{s}{p} \} \right)$. (This choice is possible by Assumption \ref{Assumption:Intro:Exp}). Let $B = B_R(x_0)$ be a ball with $R \leq 1$. Define $H(x,y,U) := G(x,y,U)^{(p-1)/p}$. Then there exists a constant $C$ depending only on $\texttt{data}$ such that for any solution $u \in W^{s,p}(\bbR^n) \cap L^{\infty}(\bbR^n)$ to \eqref{eq:Intro:MainEqn} and for any $\sigma \in (0,1)$
\begin{equation}\label{eq:RevHolderFinal}
\begin{split}
\left( \fint_{\frac{1}{4} \cB} H(x,y,U)^{p'} \, \rmd \nu \right)^{1/p'} &\leq \frac{C}{\sigma \veps^{1/\gamma - 1/p'}} \left( \fint_{\cB} H(x,y,U)^{\gamma} \, \rmd \nu \right)^{1/\gamma} \\
	&\qquad + \frac{C \sigma}{ \veps^{1/\gamma - 1/p'}} \sum_{k=0}^{\infty} \alpha_k \left( \fint_{2^k \cB} H(x,y,U)^{\gamma} \, \rmd \nu \right)^{1/\gamma} \\
	&+ \frac{C [\nu(\cB)]^{\theta}}{\veps^{1/p_* - 1/p'}} \left( \fint_{\cB} F^{p_*} \, \rmd \nu \right)^{1/p_*}\,, \\
\end{split}
\end{equation}
where $\gamma := \frac{\eta}{p-1} =  p' \cdot  \frac{n+\veps p}{n+sp+\veps p} < p'$ and
$
\theta := \frac{s-\veps(p-1)}{n+\veps p}
$.
\end{corollary}

\begin{proof}
Apply \eqref{eq:SeriesComparisons}-\eqref{eq:DefnOfGeometricSeries} and the pointwise inequality $U^{\eta} \leq (U^p + A U^q)^{\eta/p} = H^{\eta/(p-1)} = H^{\gamma}$ to each integral on the right-hand side of \eqref{eq:RevHolder}.
The result then follows by raising both sides of  \eqref{eq:RevHolder} to the power $p-1$, using the estimate $(a+b+c)^{p-1} \leq 3^{p-2}(a^{p-1}+b^{p-1}+c^{p-1})$ on the right-hand side, and finally using H\"older's inequality on the infinite sum. 
That is, if we set $b_k =  \left( \fint_{2^k \cB} H^{\gamma} \, \rmd \nu \right)^{1/\eta}$,
\begin{equation*}
\left( \sum_{k=0}^{\infty} \alpha_k b_k \right)^{p-1} \leq \left( \sum_{k=0}^{\infty} \alpha_k b_k^{p-1} \right) \left( \sum_{k=0}^{\infty} \alpha_k \right)^{p-2}\,.
\end{equation*}
\end{proof}

\begin{remark}
	If $a \equiv 0$ one can see from careful inspection of the proofs they need not assume $u \in L^{\infty}(\bbR^n)$ in Proposition \ref{prop:RevHolder} and Corollary \ref{cor:RevHolderFinal}.
\end{remark}

\section{Fractional Gehring Lemma}\label{sec:GehringLemma}
We restate the inequalities that $\veps > 0$ need to satisfy 
\begin{equation}\label{Assumption:Epsilon}
\veps \in \left( 0, \frac{s}{p} \right)\,, \qquad \veps < s \left( \frac{tq}{sp} - \frac{1}{p'} \right)\,, \qquad \veps < 1 - s \,,
\end{equation}
so that the results of previous sections hold. We also recall $H(x,y,U) := G(x,y,U)^{(p-1)/p}$ as given in Corollary \ref{cor:RevHolderFinal}.

\begin{theorem}\label{thm:Gehring:MainThm}
Let $u \in W^{s,p}(\bbR^n) \cap L^{\infty}(\bbR^n)$ be a bounded weak solution to \eqref{eq:Intro:MainEqn} with  Assumptions \eqref{Assumption:Intro:Coeff} and \eqref{Assumption:Intro:Exp}. Define $F$ as in \eqref{eq:DefnOfUandF} with $f \in L^{p_* + \delta_0}(\bbR^n)$ for given $\delta_0 > 0$. 
Then there exists a constant $\veps \in (0,1-s)$ depending on $\texttt{data}$ and $\delta_0$, and constants $\delta \in (0,1)$ and $C_1$ depending on $\texttt{data}$ and $\veps$ such that whenever $\cB \equiv \cB(x_0,\varrho_0) \subset \bbR^{2n}$ with $\varrho_0 \leq 1$ we have
\begin{equation}\label{eq:Gehring:Mainthm:SelfImprovementInequality}
\begin{split}
\left( \fint_{\cB} H(x,y,U)^{p'+\delta} \, \rmd \nu \right)^{1/(p'+\delta)} &\leq C_1 \sum_{k=0}^{\infty} 2^{-k(\frac{tq}{p-1} - s - \veps)} \left( \fint_{2^k \cB} H(x,y,U)^{p'} \, \rmd \nu \right)^{1/p'} \\
	&+ C_1 \left( \fint_{\cB} F^{p_*+\delta_0} \, \rmd \nu \right)^{1/(p_*+\delta_0)}\,.
\end{split}
\end{equation}
\end{theorem}

\begin{proof}
Note that all quantities on the right-hand side of \eqref{eq:Gehring:Mainthm:SelfImprovementInequality} are finite following the argument in Remark \ref{remark-finiteness-sums}. 
Define the truncated function $H_m = \min \{ H,m \}$ for positive integers $m$, and define the measure $\rmd \mu = H^{p'} \rmd \nu$. Choose $\alpha$ and $\beta$ such that $\varrho_0 < \beta < \alpha < 2 \varrho_0 $, so that
\begin{equation*}
\cB(x_0,\varrho_0) \subset \cB(x_0,\beta) \subset \cB(x_0,\alpha) \subset \cB(x_0 ,2 \varrho_0)\,.
\end{equation*}
By using the distributional form of the integral, 
\begin{equation}\label{eq:MainProof:BeginningEstimate}
\begin{split}
\int_{\cB_{\beta}} H_m^{\delta} H^{p'} \, \rmd \nu &= \int_{\cB_{\beta}} H_m^{\delta} \, \rmd \mu \\
	&= \delta \intdmt{0}{\infty}{\lambda^{\delta -1 } \mu \left( \cB_{\beta} \cap \{ H_m > \lambda \} \right) }{\lambda} \\
	&= \delta \intdmt{0}{m}{\lambda^{\delta -1 } \intdm{\cB_{\beta} \cap \{ H > \lambda \}}{H^{p'}}{\nu}}{\lambda} \\
	&\leq \lambda_0^{\delta} \int_{\cB_{\beta}} H^{p'} \, \rmd \nu + \delta \intdmt{\lambda_0}{m}{\lambda^{\delta -1 } \intdm{\cB_{\beta} \cap \{ H > \lambda \}}{H^{p'}}{\nu}}{\lambda} \\
	&:= \rmI + \mathrm{II}\,,
\end{split}
\end{equation}
where $\lambda_0 > 0$ is a constant. We define it here as
\begin{equation}\label{eq:DefinitionOfLambda}
\begin{split}
\lambda_0 := \frac{C_a}{\veps} \left( \frac{\varrho_0}{\alpha-\beta} \right)^{2n+p} \big\{ \Upsilon_0(x_0, 2 \varrho_0) + Tail(x_0, 2 \varrho_0) + \Psi_1(x_0, 2 \varrho_0)  \big\}\,,
\end{split}
\end{equation}
where the constant $C_a$ depends only on $\texttt{data}$,
\begin{equation}\label{eq:DefinitionOfTails1}
\begin{split}
\Upsilon_0(x_0,R) &:= \left( \fint_{\cB(x_0,R)} F^{p_*+\delta_f} \, \rmd \nu \right)^{1/(p_* + \delta_f)}\,, \text{ with } \delta_f \in (0,\delta_0) \text{ to be determined,} \\
Tail(x_0,R) &:= \sum_{k=0}^{\infty} 2^{-k(\frac{tq}{p-1}-s-\veps)} \left( \fint_{\cB(x_0,2^k R)} H^{\gamma} \, \rmd \nu \right)^{1/\gamma}\,,
\end{split}
\end{equation}
and, for any constant $M \geq 1$,
\begin{equation}\label{eq:DefinitionOfTails2}
\Psi_M(x_0,R) := \left( \fint_{\cB(x_0,R)} H^{p'} \, \rmd \nu \right)^{1/p'} + M \frac{[\nu(\cB(x_0,R))]^{\theta}}{\veps^{1/p_* - 1/p'}} \left( \fint_{\cB(x_0,R)} F^{p_*} \, \rmd \nu \right)^{1/p_*}\,;
\end{equation}
we write $\Psi_M$ with $M=1$ as $\Psi_1$. The definition of this constant $\lambda_0$ is motivated by the right-hand side of \eqref{eq:RevHolderFinal}.
With this choice of $\lambda_0$, the first term in \eqref{eq:MainProof:BeginningEstimate} is easily estimated as
\begin{equation}\label{eq:MainProof:EstimateOfI}
\rmI \leq \lambda_0^{\delta} \, \nu(\cB_{2 \varrho_0}) \fint_{\cB_{2 \varrho_0}} H^{p'} \, \rmd \nu \leq C \nu(\cB_{\varrho_0}) \lambda_0^{p'+\delta}\,,
\end{equation}
by the definition of $\lambda_0$ and by the doubling property \eqref{eq:MeasureDoublingProperty}, with $C = C(\texttt{data}, \veps)$.

The constant $\lambda_0$ is chosen to additionally estimate the $\mu$-measure of the level set $\{ H > \lambda \}$ that appears in $\mathrm{II}$. It turns out that for every $\lambda \geq \lambda_0$
\begin{equation}\label{eq:MainProof:LevelSetEstimate}
\frac{1}{\lambda^{p'}} \mu \big( \cB(x_0,\beta) \cap \{ H > \lambda \} \big) \leq \frac{C_{\alpha}}{\veps^{\vartheta} \lambda^{\gamma}} \intdm{\cB(x_0,\alpha) \cap \{ H > \lambda \} }{H^{\gamma}}{\nu} + \frac{C_f \lambda_0^{\vartheta_f}}{\lambda^{\widetilde{\vartheta}_f}} \intdm{\cB(x_0,\alpha) \cap \{F > \kappa_f \lambda \} }{F^{p_*}}{\nu}\,,
\end{equation}
%
where constants $C_{\alpha}(\texttt{data}) > 0$, $C_f(\texttt{data},\veps) \geq 1$, $\kappa_f(\texttt{data},\veps) \in (0,1)$, and positive constants
\begin{equation*}
\vartheta := 
\frac{3(p'-\gamma)}{\gamma}\,,
	\qquad
\vartheta_f := 
(p_* + \delta_f) \left( \frac{p_* \theta}{1-p_* \theta} \right)\,,
	\qquad 
\widetilde{\vartheta}_f := 
\frac{p_*(1+\theta \delta_f)}{1-p_* \theta}\,.
\end{equation*}
The reverse H\"older inequality \eqref{eq:RevHolderFinal} is used to prove this level set estimate. The proof is quite technical. For $p=2$ this same level set estimate is proved in \cite[Section 5]{kuusi2015} with appropriately defined constants and the proof for $p>2$ can be carried out in almost exactly the same manner adjusting the constants to fit into the new setup. Presenting its proof in this work will force us to repeat arguments from \cite[Section 5]{kuusi2015}. Instead we have chosen to present the proof in the companion note \cite{Scott-Mengesha-Levelset} for the sake of completeness.  For now, we use \eqref{eq:MainProof:LevelSetEstimate} to estimate $\mathrm{II}$, and obtain
\begin{equation}\label{eq:MainProof:FirstEstimateOfII}
\begin{split}
\mathrm{II} &\leq \delta \intdmt{\lambda_0}{m}{\lambda^{\delta -1 } \left( \frac{C_{\alpha}}{\veps^{\vartheta} \lambda^{\gamma- p'}} \intdm{\cB(x_0,\alpha) \cap \{ H > \lambda \} }{H^{\gamma}}{\nu} + \frac{C_f \lambda_0^{\vartheta_f}}{\lambda^{\widetilde{\vartheta}_f - p'}} \intdm{\cB(x_0,\alpha) \cap \{F > \kappa_f \lambda \} }{F^{p_*}}{\nu} \right) }{\lambda} \\
	&\leq \frac{C_{\alpha} \delta}{\veps^{\vartheta}} \intdmt{0}{\infty}{ \lambda^{p' - \gamma - 1 + \delta} \intdm{\cB_{\alpha} \cap \{ H_m > \lambda \}}{H^{\gamma}}{\nu}}{\lambda} + C_f \delta \int_{\lambda_0}^m \frac{\lambda_0^{\vartheta_f}}{\lambda^{\widetilde{\vartheta}_f-\delta-p'+1}} \intdm{\cB_{\alpha} \cap \{F > \kappa_f \lambda \} }{F^{p_*}}{\nu} \, \rmd \lambda\\
	&:= \mathrm{II}_1 + \mathrm{II}_2\,.
\end{split}
\end{equation}
To estimate $\mathrm{II}_1$, we choose $\delta > 0$ to satisfy
\begin{equation}\label{eq:DeltaCondition1}
\frac{C_{\alpha}(n+sp+p)}{spp'} \cdot \frac{\delta}{\veps^{3sp/n}} < 1/4\,,
\end{equation}
so that
\begin{equation*}
\frac{C_{\alpha}}{(p'-\gamma) \veps^{\vartheta}} \delta \leq \frac{C_{\alpha}(n+sp+\epsilon p)}{spp'} \cdot \frac{\delta}{\veps^{3sp/n}}  < 1/4\,,
\end{equation*}
and so using Fubini's Theorem
\begin{equation}\label{eq:MainProof:EstimateOfII1}
\mathrm{II}_1 = \frac{C_{\alpha} \delta}{(p'-\gamma+\delta) \veps^{\vartheta}} \intdm{\cB_{\alpha}}{H_m^{p'-\gamma+\delta} H^{\gamma}}{\nu} \leq \frac{1}{4} \intdm{\cB_{\alpha}}{H_m^{p'-\gamma+\delta} H^{\gamma}}{\nu} \leq \frac{1}{4} \intdm{\cB_{\alpha}}{H_m^{\delta} H^{p'}}{\nu} \,.
\end{equation}

To estimate $\mathrm{II}_2$
we need an additional condition on $\veps$, and in turn an additional condition on $\delta$. Recall that $\veps > 0$ can be as small as we wish, but it has not been fixed yet. We do this now. Note that for $\veps > 0$
\begin{equation*}
\frac{\veps p (p')^2}{n+\veps p} < \frac{\veps p (n + s p')}{n(s-\veps(p-1))}\,.
\end{equation*}
Therefore, we can find $\veps > 0$ satisfying \eqref{Assumption:Epsilon} as well as a number $\delta_f \in (0,\delta_0)$ sufficiently small such that
\begin{equation}\label{eq:DeltaFCondition1}
\frac{\veps p (p')^2}{n+\veps p} < \delta_f \leq \frac{\veps p (n + s p')}{n(s-\veps(p-1))}\,.
\end{equation}
Now that the positive constants $\veps$ and $\delta_f$ have been fixed, we come to our second condition on $\delta$, namely the upper bound
\begin{equation}\label{eq:DeltaCondition2}
\delta \leq  \frac{1}{p-1} \cdot \frac{\veps p p'  (n+s p')}{n^2 + 2n \veps p + \veps s p p'} \,.
\end{equation}
 a consequence of this assumption is the bound
\begin{equation}\label{eq:MainProof:ConsequenceOfDeltaCond}
\delta \leq \delta_f \left( \frac{(n+\veps p) (n+s p')}{n^2 + 2n \veps p + \veps s p p'} \right) - \frac{\veps p p' (n+s p')}{n^2 + 2n \veps p + \veps s p p'}\,.
\end{equation}
Indeed, using the lower bound in \eqref{eq:DeltaFCondition1},
\begin{equation*}
\begin{split}
\delta \leq \frac{1}{p-1} \cdot \frac{\veps p p'  (n+s p')}{n^2 + 2n \veps p + \veps s p p'} &= \frac{\veps p (p')^2}{n+\veps p} \cdot \frac{(n+\veps p) (n+s p')}{n^2 + 2n \veps p + \veps s p p'} - \frac{\veps p p' (n+s p')}{n^2 + 2n \veps p + \veps s p p'} \\
	&\leq \delta_f \left( \frac{(n+\veps p) (n+s p')}{n^2 + 2n \veps p + \veps s p p'} \right) - \frac{\veps p p' (n+s p')}{n^2 + 2n \veps p + \veps s p p'}\,.
\end{split}
\end{equation*}

With these assumptions, we return to estimating $\mathrm{II}_2$. By changing variables and using Fubini's theorem,
\begin{equation*}
\begin{split}
\int_{\lambda_0}^m \lambda^{\delta - 1 + p' - \widetilde{\vartheta}_f } \intdm{\cB_{\alpha} \cap \{ F > \kappa_f \lambda \} }{F^{p_*}}{\nu} \, \rmd \lambda 
	&\leq C \int_0^{\infty} \lambda^{\delta - 1 + p' - \widetilde{\vartheta}_f} \intdm{\cB_{\alpha} \cap \{ F > \lambda \} }{F^{p_*}}{\nu} \, \rmd \lambda \\
	&\leq \frac{C \nu(\cB_{2 \varrho_0})}{\delta + p' - \widetilde{\vartheta}_f} \fint_{\cB_{2 \varrho_0}} F^{p_* + \delta + p'  - \widetilde{\vartheta}_f} \, \rmd \nu \\
	&\leq \frac{C \nu(\cB_{2 \varrho_0})}{\delta} \fint_{\cB_{2 \varrho_0}} F^{p_* + \delta + p' - \widetilde{\vartheta}_f} \, \rmd \nu\,, \\
\end{split}
\end{equation*}
where $C = C(\texttt{data}, \veps)$. In the last inequality we used \eqref{eq:DeltaFCondition1} and that 
\begin{equation*}
 \delta_f \leq \frac{\veps p (n+s p')}{n(s-\veps(p-1))} \quad \Leftrightarrow \quad p' - \widetilde{\vartheta}_f \geq 0\,.
\end{equation*}
The integral in the last inequality is finite so long as $p_* + \delta + p' - \widetilde{\vartheta}_f \leq p_* + \delta_f$, but this is equivalent to \eqref{eq:MainProof:ConsequenceOfDeltaCond}. Therefore by H\"older's inequality
\begin{equation}\label{eq:MainProof:EstimateOfII2}
\begin{split}
\mathrm{II}_2 &\leq C  \nu(\cB_{2 \varrho_0}) \lambda_0^{\vartheta_f} \fint_{\cB_{2 \varrho_0}} F^{p_* + \delta + p' - \widetilde{\vartheta}_f} \, \rmd \nu \\
	&\leq C  \nu(\cB_{2 \varrho_0}) \lambda_0^{\vartheta_f} \left( \fint_{\cB_{2 \varrho_0}} F^{p_* + \delta_f} \, \rmd \nu \right)^{\frac{p_* + \delta + p' - \widetilde{\vartheta}_f}{p_* + \delta_f}} \\
	&= C \nu(\cB_{2 \varrho_0}) \lambda_0^{\vartheta_f} [\Upsilon_0(x_0,2 \varrho_0) ]^{p_* + \delta + p' - \widetilde{\vartheta}_f} \leq C \nu(\cB_{\varrho_0}) \lambda_0^{\vartheta_f + p_* + \delta + p' - \widetilde{\vartheta}_f} = C \nu(\cB_{\varrho_0}) \lambda_0^{p' + \delta}\,,
\end{split}
\end{equation}
where we additionally used \eqref{eq:MeasureDoublingProperty}. The constant $C$ depends on $\texttt{data}$ and $\veps$.

%
Combining \eqref{eq:MainProof:EstimateOfI}, \eqref{eq:MainProof:EstimateOfII1} and \eqref{eq:MainProof:EstimateOfII2} in the estimate \eqref{eq:MainProof:BeginningEstimate} gives
\begin{equation}\label{eq:MainProof:BeginningOfIteration}
\intdm{\cB_{\beta}}{H_m^{\delta} H^{p'}}{\nu} \leq \frac{1}{4} \intdm{\cB_{\alpha}}{H_m^{\delta} H^{p'}}{\nu} + C \nu(\cB_{\varrho_0}) \lambda_0^{p'+\delta}\,.
\end{equation}
Therefore, using the doubling property \eqref{eq:MeasureDoublingProperty} and using the definition of $\lambda_0$ in \eqref{eq:DefinitionOfLambda},
\begin{equation*}
\begin{split}
\left( \frac{\nu(\cB_{\beta})}{\nu(\cB_{\alpha})} \fint_{\cB_{\beta}} H_m^{\delta} H^{p'} \, \rmd \nu \right)^{1/(p'+\delta)} &\leq \left( \frac{1}{4}  \fint_{\cB_{\alpha}}H_m^{\delta} H^{p'} \, \rmd \nu \right)^{1/(p'+\delta)} + C \left(  \frac{\nu(\cB_{\varrho_0})}{\nu(\cB_{\alpha})} \right)^{1/(p'+\delta)} \lambda_0 \\
	&\leq \frac{1}{2} \left( \fint_{\cB_{\alpha}} H_m^{\delta} H^{p'} \, \rmd \nu \right)^{1/(p'+\delta)} + \frac{C}{\veps} \left( \frac{\varrho_0}{\alpha-\beta} \right)^{2n+p} \Theta(x_0,2 \varrho_0)\,.
\end{split}
\end{equation*}
where
\begin{equation}\label{eq:DefnOfTheta}
\Theta(x_0,R) := \Upsilon_0(x_0, R) + Tail(x_0, R) + \Psi_1(x_0,R)\,.
\end{equation}
We can rewrite the above inequality as
\begin{equation*}
\varphi(\beta) \leq \frac{1}{2} \varphi(\alpha) + \frac{C}{\veps} \left( \frac{\varrho_0}{\alpha-\beta} \right)^{2n+p} \Theta(x_0,2 \varrho_0)\,,
\end{equation*}
where $\varphi(\varrho) := \left( \fint_{\cB_{\varrho}(x_0)} H_m^{\delta} H^{p'} \, \rmd \nu \right)^{1/(p'+\delta)}$ for $\varrho \in [\varrho_0, \frac{3}{2} \varrho_0]$. Therefore, by an iteration lemma
\cite[Chapter 6, Lemma 6.1]{alma9916407080102311} we come to
\begin{equation*}
\left( \fint_{\cB_{\varrho_0}(x_0)} H_m^{\delta} H^{p'} \, \rmd \nu \right)^{1/(p'+\delta)} = \varphi(\varrho_0) \leq C \Theta (x_0,2 \varrho_0)\,,
\end{equation*}
where $C = C(\texttt{data},\veps)$ is independent of $m$. Therefore, we can take $m \to \infty$ and by Fatou's Lemma obtain
\begin{equation*}
\left( \fint_{\cB_{\varrho_0}(x_0)} H^{p'+\delta} \, \rmd \nu \right)^{1/(p'+\delta)} \leq C \Theta(x_0,2 \varrho_0)\,.
\end{equation*}
The result \eqref{eq:Gehring:Mainthm:SelfImprovementInequality} follows by recalling the definition of $\Theta$ and using H\"older's inequality.
\end{proof}

\begin{proof}[Proof of Theorem \ref{thm:Intro:MainThm}]
The result follows by using the definitions of $H$, $G$ and $P$. Using the $\delta>0$ from Theorem \ref{thm:Gehring:MainThm}, we have for any ball $B_R(x_0)$ with radius $R \leq 1$
\begin{equation*}
\begin{split}
\int_B \int_B [P(x,y,u(x),u(y)]^{1+ \frac{p-1}{p} \delta}  |x-y|^{\frac{p-1}{p} \delta(n-\veps)} \, \rmd y \, \rmd x &= \int_B \int_B  G^{1+ \frac{p-1}{p} \delta} \, \rmd \nu=
\int_B \int_B H^{p'+\delta} \, \rmd \nu  < \infty. 
\end{split}
\end{equation*}
As a consequence, we see that $P(x,y,u(x),u(y))$ is in the weighted space $L^{1+ \frac{p-1}{p} \delta}_{\omega} (B\times B)$ where the weight $\omega(x, y) = |x-y|^{\frac{p-1}{p} \delta(n-\veps)}$. A simple computation shows that $\omega$ is a Muckenhoupt $A_{1+ \frac{p-1}{p} \delta}(\mathbb{R}^{2n})$. Thus using reverse H\"older property of  Muckenhoupt weights, see \cite[Corollary 3.3]{AdimurthiKarthik2018GWNI}  there exists a $\tau>0$ such that 
\[
P(x,y,u(x),u(y)) \in L^{1+\tau}(B\times B), 
\]
and therefore, $P \in L^{1+ \tau}_{loc}(\bbR^{2n})$ via a covering argument. Moreover,
\begin{equation*}
\begin{split}
\infty > \int_B \int_B  G^{1+ \frac{p-1}{p} \delta} \, \rmd \nu \geq \int_B \int_B  (U^p)^{1+ \frac{p-1}{p} \delta} \, \rmd \nu  &= \int_B \int_B \frac{|u(x)-u(y)|^{p + (p-1)\delta}}{|x-y|^{n+sp+s\delta(p-1)+\veps(p-1)\delta}}\, \rmd y \, \rmd x\,.\end{split}
\end{equation*}
Rewriting the last integral,
\begin{equation*}
\int_B \int_B \frac{|u(x)-u(y)|^{p + (p-1)\delta}}{|x-y|^{n+(p+(p-1)\delta) [ s + \frac{\veps (p-1) \delta}{p+(p-1)\delta} ] }}\, \rmd y \, \rmd x < \infty\,,
\end{equation*}
so thus $u \in W^{s + \frac{\veps (p-1) \delta}{p+(p-1)\delta}, p + (p-1)\delta}_{loc}(\bbR^n)$ by a similar covering argument. Note that since $\veps < 1-s$ the differentiability exponent $s + \frac{\veps (p-1) \delta}{p+(p-1)\delta} < 1$. The definitions of the constants $\veps_0$, $\veps_1$ and $\veps_2$ now follow by inspection of the proof.
\end{proof}


\end{document}